\documentclass[11pt]{article} 
\usepackage{a4,amssymb, amsmath, amsthm, dsfont}
\usepackage{mathrsfs,amsfonts}
\usepackage{graphics,color}
\usepackage{epsfig}

\newtheorem{theorem}{\bf Theorem}[section]
\newtheorem{definition}{\bf Theorem}[section]
\newtheorem{remark}{\bf Theorem}[section]

\newtheorem{corollary}{\bf Corollary}[section]

\newtheorem{lemma}{\bf Lemma}[section]
\def\be{\begin{equation}}
	\def\ee{\end{equation}}
\def\bse{\begin{subequations}}
	\def\ese{\end{subequations}}
\def\bge{\begin{eqnarray}}
	\def\bgee{\begin{eqnarray*}}
		\def\ege{\end{eqnarray}}
	\def\egee{\end{eqnarray*}}

\usepackage{mathtools}
\usepackage{amsfonts}
\usepackage{dsfont}
\usepackage{array}
\begin{document}

\title{Wellposedness of an Elliptic-Dispersive Coupled System for MEMS}

\author{
Heiko Gimperlein\thanks{Leopold-Franzens-Universit\"{a}t Innsbruck, Engineering Mathematics, Technikerstra\ss e 13, 6020 Innsbruck, Austria} \and  Runan He\thanks{Institut f\"{u}r Mathematik, Martin-Luther-Universit\"{a}t Halle-Wittenberg, 06120 Halle (Saale), Germany} \and  Andrew A.~Lacey\thanks{Maxwell Institute for Mathematical Sciences and Department of Mathematics, Heriot-Watt University, Edinburgh, EH14 4AS, United Kingdom}}
\date{}

\maketitle \vskip 0.5cm
\begin{abstract}
\noindent  In this work, we study the local wellposedness of the solution to a nonlinear elliptic-dispersive coupled system which serves as a model for a Micro-Electro-Mechanical System (MEMS). A simple electrostatically actuated MEMS capacitor device consists of two parallel plates separated by a gas-filled thin gap.  The nonlinear elliptic-dispersive coupled system modelling the device combines a linear elliptic equation for the gas pressure with a semilinear dispersive equation for the gap width. We show the local-in-time existence of strict solutions for the system, by combining  elliptic regularity results for the elliptic equation, Lipschitz continuous dependence of its solution on that of the dispersive equation, and then local-in-time existence for a resulting abstract dispersive problem.  Semigroup approaches are key to solve the abstract dispersive problem.
\end{abstract}

\section{Introduction} \label{Chap5Sec:Intro}
This paper is concerned with short-time existence, uniqueness and smoothness of the solution to the following nonlinear coupled system, which is a model for an idealized electrostatically actuated MEMS device containing an effectively incompressible fluid:
\bse\label{Chap5cp2}
\be\label{Chap5cp2-1-1}
\frac{\partial w}{\partial t}=\nabla\cdot\left(w^3\nabla u\right),\quad x\in\Omega,\ t\geq 0;
\ee
\be\label{Chap5cp2-1-2}
\frac{\partial^2w}{\partial t^2}=\Delta w-\Delta^2w-\frac{\beta_F}{w^2}+\beta_p(u-1),\quad x\in\Omega,\ t\geq 0;
\ee
\be\label{Chap5cp2-2}
w(x,0)=w_0(x),\ \frac{\partial w}{\partial t}(x,0)=v_0(x),\quad x\in {\Omega};
\ee
\be\label{Chap5cp2-3}
u(x,t)=\theta_1, \ w(x,t)=\theta_2, \ \Delta w(x,t)=0,\quad x\in\partial\Omega,\ t\geq 0.
\ee
\ese
Here $u(x,t)$ and $w(x,t)$ are the unknown functions, corresponding to gas pressure and gap width respectively,  $\Omega\subset\mathbb{R}^n$ is a bounded and open region with smooth boundary $\partial\Omega$, $n=1,\ 2$; $\beta_F>0$, $\beta_p>0$, $\theta_1>0$ and $\theta_2>0$ are  given constants; $v_0=v_0(x)$ and $w_0=w_0(x)$ are  given functions. We shall prove the following wellposedness for short time:
\begin{theorem}\label{Chap5coupled system}
Let $\Omega\subseteq\mathbb{R}^n$ with smooth boundary, $n=1,\ 2$. For sufficiently smooth, positive initial values  compatible with the boundary conditions, there exists a time $T_0>0$ such that the initial-boundary value problem  \eqref{Chap5cp2} has a unique strict solution $(u, w)$, and
\[u\in C\left([0, T_0); H^1(\Omega)\right),\ w\in C^{2}\left([0, T_0); L^2(\Omega)\right)\cap C^{1}([0, T_0); H^2(\Omega))\cap C\left([0, T_0); H^4(\Omega)\right).\]
\end{theorem}
The explicit assumptions on the initial conditions will be specified in Section \ref{Chap54th-order problem}. Note that global-in-time solutions are not necessarily expected, as quenching singularities with $\displaystyle\inf_{\Omega} w(t)\to 0$ may develop in finite time \cite{GHL22}.
\\
The model \eqref{Chap5cp2} describes the behaviour of a simple electrically actuated MEMS (Micro-Electro-Mechanical Systems) capacitor (see, for example, \cite{JS}). Such a device contains two conducting plates which, when the device is uncharged and at equilibrium, are close and parallel to each other. More generally, we suppose a fixed potential difference is applied; this potential difference acts across the plates, so that the MEMS device forms a capacitor. The two plates lie inside a sealed box also containing a {rarefied but effectively incompressible} gas, as opposed to a perfect vacuum. This gas then serves to give a small resistance to the motion of the upper of the plates, this plate being flexible but pinned around its edges. The other, lower, plate is taken to be rigid and flat.
	
	 Eqn.~{\eqref{Chap5cp2-1-1}} is the usual version of the standard Reynolds' equation, for the flow of an incompressible fluid through a narrow gap (width $w$) driven by pressure $u$  
(see, for example, {\cite{OO}}).  
Here it is assumed that the squeeze film action is very slow, so the gas has pressure and density nearly constant, and flow can thus be regarded as incompressible. 
	The paper \cite{BY} studies the incompressible squeeze film equation {\eqref{Chap5cp2-1-1}} by  linearizing further, for cases where deflection of the plate from its equilibrium position is small ($w =$ const., to leading order), and thinking of it as Poisson's equation $\Delta u = \frac{\partial w}{\partial t}$.
	
	The upper part of the capacitor acts as a thin elastic plate so that its motion can be taken to be given by a dynamic plate equation balancing the inertial term on the left-hand side of {\eqref{Chap5cp2-1-2}} with the biharmonic term modelling linear elasticity, the second term on the right {\cite{HKO}}. Taking there to be a significant tension acting across the plate gives, additionally, the first term on the right: the more usual Laplacian, as for a membrane. The third term is the electrostatic force attracting the upper plate towards the lower. The strength of this force per unit area is given by the local electric field strength times the surface charge density, the latter itself being proportional to the former, while this, the field strength, is inversely proportional to the gap width $w$. The final, fourth, term is simply the net upward gas pressure acting on the plate: pressure in the gap acting up and constant ambient pressure acting down. For more details see {\cite{EGG}, \cite{BP}, \cite{JS}}. 

\

Our proof of Theorem \ref{Chap5coupled system} relies on semigroup techniques developed for semilinear hyperbolic equations. Using the solution operator for the first, elliptic equation, we reduce the system to an abstract evolution equation, to which the functional analytic methods are then applied. Semigroup methods have become a powerful tool for MEMS-related models defined by a single equation or by an elliptic-parabolic coupled system, see the recent survey \cite{LW}.

	Before proceeding, we review some of the most relevant literature which studies related models MEMS devices, both numerically and analytically, to obtain qualitative behaviour, and note some other work on similar systems. 
	
	The static deflection of charged elastic plates in electrostatic actuators can be represented by a nonlinear elliptic equation
	\begin{equation}\label{BendingEqn3}
		-\Delta w+\beta_e\Delta^2 w=-\frac{\beta_F}{w^2}, 
	\end{equation}
where $\beta_F$ is an electrostatic coefficient, depending upon applied voltage and giving the strength of attraction between the plates, and $\beta_e$ is fixed by the elastic moduli of the plate material in comparison with imposed tension. 
	Lin et al.~\cite{LY} study the existence, construction, approximation, and behaviour of classical and singular solutions to equation \eqref{BendingEqn3}. Other such problems can be found in references  \cite{EG}, \cite{VM}. 	From Chapter 12 in \cite{EGG}, there is a value $\beta^*\in(0, \infty)$  such that for $0<\beta_F<\beta^*$ there exists at least one weak solution to \eqref{BendingEqn3}, while no solution exists for $\beta_F>\beta^*$.

	To model the behaviour of the plate gap width $w$ over time, one has to consider the momentum of the plate as it deforms and damping forces, as well as the elastic nature of the plate and the electrostatic force. Then a single equation of motion which has been studied is
	\begin{equation}\label{MotEqn}
		\epsilon^2\frac{\partial^2w}{\partial t^2}+\frac{\partial w}{\partial t}-\Delta w+\beta_e\Delta^2 w=-\frac{\beta_F}{w^2}.
	\end{equation}
	\[ \epsilon^2 = \frac{\text{inertial\ coefficient}}{\text{damping\ coefficient}}, \]
the second term in (\ref{MotEqn}) having been introduced to account for damping.
	
	For a physical model, the equation applies on a bounded domain of $\mathbb{R}^n$, $1\leq n\leq 3$.
	In Guo, \cite{Guo}, quenching is seen to occur, that is, $w$ falls to zero in a finite time, so that a solution of (\ref{MotEqn}) will cease to exist, if $\beta_F$ exceeds the critical value $\beta^*$ for the time-independent problem \eqref{BendingEqn3}. This creates what, physically, is called the pull-in instability, when the two plates touch (also referred to as ``touch-down''). Guo \cite{Guo} shows that there exists a $\beta_{F_{1}}\in(0,\ \beta^*]$ such that for $0\leq \beta_F<\beta_{F_{1}}$, the solution of an initial boundary value problem for \eqref{MotEqn} globally exists. Under some further technical hypotheses, in this case the solution exponentially converges to a regular steady state.
See the survey article \cite{LW} for a discussion of a wider class of models arising in the description of MEMS.	\\

\noindent The plan of the paper is as follows: In Section 2, we introduce the relevant function spaces and some of their basic properties. In Section 3, we establish auxiliary  estimates for the nonlinearity in \eqref{Chap5cp2-1-2} and introduce the notions of both mild and strict solutions of  evolution equations. The elliptic problem \eqref{Chap5cp2-1-1} for the unknown $u$ is discussed in Section 4. The results from Sections 3 and 4 are then used in Section 5 to investigate the local well-posedness of the initial-boundary value problem \eqref{Chap5cp2}, leading to Theorem \ref{Chap5coupled system}.

\section{Notation}\label{SecNotations}
Denote by $\beta_F$, $\beta_p$, $\theta$ given positive constants. Let $\Omega\subset\mathbb{R}^n$ be an open and bounded subset with smooth boundary $\partial\Omega$, $n=1,\ 2$. Denote by $C=C(\Omega)$ a positive constant which may vary from line to line below but only depends on $\Omega$.
\be\label{B-0}
H^4_o(\Omega)=\left\{\psi\in H^4(\Omega): \psi|_{\partial\Omega}=\Delta\psi\mid_{\partial\Omega}=0\right\},\quad H^2_o(\Omega)=H^{2}(\Omega)\cap H^1_{0}(\Omega),
\ee
\be\label{B-1}
H^4_{\theta}(\Omega)=\left\{\psi\in H^4(\Omega): \psi(x)=\theta,\ \Delta\psi(x)=0,\ {x\in\partial\Omega}\right\},
\ee
\be\label{B2-1}
H^2_{\theta}(\Omega)=\left\{\psi\in H^2(\Omega):\ \psi\mid_{\partial\Omega}=\theta\right\}.
\ee
\begin{definition}
	We denote by $X$ a Banach space, with norm $\|\cdot\|_{X}$. For $k\in\mathbb{N}$ and $T\in(0, \infty)$, $\mathcal{B}(X)$ denotes the space of bounded linear operators on $X$. In the following we shall be particularly interested in $X=L^2(\Omega)$, $L^\infty(\Omega)$ and $H^k(\Omega)$. The space $\mathcal{B}([0,T];X)$ consists of all  measurable, almost everywhere bounded functions $u: t\in[0,T]\longrightarrow u(t)\in X$, with norm $\|u\|_{\mathcal{B}([0,T];X)}=\mathrm{esssup}_{t\in[0, T]}\|u(t)\|_{X}$. If $X$ is a function space as above, we write $u(t): x\in\Omega\longrightarrow[u(t)](x)=u(x,t)\in\mathbb{R}$. The closed subspace of continuous functions is denoted by $C([0,T];X)$, and
	\[C^k([0,T];X)=\left\{u:[0,T]\rightarrow X: \quad \frac{d^ju}{dt^j}\in C([0,T];X),\quad 0\leq j\leq k\right\},\] \[\|u\|_{C^k([0,T];X)}=\sup_{t\in[0, T]}\sum_{j=0}^{k}\left\|\frac{d^ju(t)}{dt^j}\right\|_{X}.\]
\end{definition}

Our main results will be shown by constructing a Picard iteration in the complete metric space  $\mathcal{Z}(T)$, given by
\begin{align}
	\mathcal{Z}(T):=\bigg\{&(\tilde{v},\tilde{w})\in C\left([0,T]; L^2(\Omega)\times H_o^2(\Omega)\right):\
	\left(\tilde{v}(0), \tilde{w}(0)\right)=\left(\tilde{v}_0, \tilde{w}_0\right),\notag\\
	&\sup_{t\in[0, T]}\|(\tilde{v}(t)-\tilde{v}_0,\tilde{w}(t)-\tilde{w}_0)\|_{L^2(\Omega)\times H^{2}(\Omega)}\leq r\bigg\}\label{ini-NBD}.
\end{align}
\section{Preliminaries} \label{Sec:prerequisite}
In this section,  we recall some well-known properties of the Sobolev spaces  $H^k(\Omega)$, where $k>0$,  and formulate  some Corollaries which will be useful in the proof of the main results. 
The proofs of Corollary \ref{estimates} and Corollary \ref{Lip-G-Lem} can be found in Appendix \ref{AppA}. We then state a general existence result for  evolution equations and the regularity in time without proof.

\subsection{Algebra Property of Sobolev Spaces}

The algebra property of Sobolev spaces will be crucial in this work, see \cite{TT} for a proof.
\begin{lemma}\label{alg}
	$H^k(\Omega)$ is an algebra when $k>\frac{n}{2}$. In particular, $H^1(\Omega)$ is an algebra if $\Omega\subset\mathbb{R}$ and $H^2(\Omega)$ is an algebra if $\Omega\subset \mathbb{R}^n$, $n=1,\ 2$.
\end{lemma}

We deduce some immediate consequences.

\begin{corollary}\label{estimates}
	There exists a constant $C=C(\Omega)>0$, such that for all $r\in\left(0, \frac{\kappa}{2C}\right)$,
	$w\in B_r\left(w_0, T\right)$ satisfies the lower bound 
	\be\label{w-lower-bound}
	w(t)\geq\frac{\kappa}{2},\quad \forall\ t\in[0, T].
	\ee
	Moreover, for all $w_1$, $w_2\in B_r(w_0, T)$, there exist positive constants $C_k$, $k=1,\ 2,\ 3$, depending on $\Omega$, $\kappa$ and $\left\|w_0\right\|_{H^2(\Omega)}$, such that
	\be\label{C-a}
	\sup_{t\in[0, T]}\left\|\frac{1}{[w_1(t)]^{k}}\right\|_{H^2(\Omega)}\leq {C_1^k},\quad k=1,\ 2,\ 3,
	\ee
	\be\label{C-d}
	\sup_{t\in[0, T]}\left\|\frac{1}{[w_1(t)]^{k}}-\frac{1}{[w_2(t)]^{k}}\right\|_{H^2(\Omega)}\leq C_k\sup_{t\in[0, T]}\left\|w_1(t)-w_2(t)\right\|_{H^2(\Omega)},\ k=2,\ 3.
	\ee
\end{corollary}
\begin{corollary}\label{Lip-G-Lem}
	The operator $G$, defined  by
	\bse\label{G-def}
	\be\label{G-def1}
	G: \tilde{w}\in B_r\left(\tilde{w}_0, T\right)\longrightarrow G(\tilde{w})\in C([0, T]; H^2(\Omega)),
	\ee
	\be\label{G-def2}
	[G(\tilde{w})](t)=G(\tilde{w}(t))=-\frac{\beta_F}{[\tilde{w}(t)+\theta_2]^2}+\beta_p(\theta_1-1),
	\ee
	\ese
	has the following properties:
	\be\label{Holdercontinuous}
	\sup_{0\leq t<t+h\leq T}\left\|[G(\tilde{w})](t+h)-[G(\tilde{w})](t)\right\|_{H^{2}(\Omega)}\leq L_G\sup_{0\leq t<t+h\leq T}\left\|\tilde{w}(t+h)-\tilde{w}(t)\right\|_{H^{2}(\Omega)},
	\ee
	\be\label{Lip-G}
	\sup_{t\in[0, T]}\left\|[G(\tilde{w}_1)](t)-[G(\tilde{w}_2)](t)\right\|_{H^{2}(\Omega)}\leq L_G\sup_{t\in[0, T]}\left\|\tilde{w}_1(t)-\tilde{w}_2(t)\right\|_{H^{2}(\Omega)},
	\ee
	\be\label{Lip-G-1}
	\sup_{t\in[0, T]}\left\|[G(\tilde{w}_1)](t)-G(\tilde{w}_0)\right\|_{H^{2}(\Omega)}\leq L_G r.
	\ee
	Here $L_G=L_G\left(\Omega,\ \kappa,\ \|w_0\|_{H^{2}(\Omega)},\ \beta_F\right)$ is a constant.\\
	Furthermore, the Fr\'{e}chet derivative $G'(\tilde{w})$ of $G(\tilde{w})$ depending on  $\tilde{w}\in B_r\left(\tilde{w}_0, T\right)$, 	\bse\label{Fre-G-def-w}
	\be\label{Fre-G-def-w1}
	G'\left(\tilde{w}\right):q\in C\left([0, T]; H_o^2(\Omega)\right)\longrightarrow G'\left(\tilde{w}\right)q\in C\left([0, T]; H_o^2(\Omega)\right),\
	\ee
	\be\label{Fre-G-def-w2}
	\left[G'\left(\tilde{w}\right)q\right](t)=\left[G'\left(\tilde{w}(t)\right)\right]q(t)=\frac{2\beta_F}{\left(\tilde{w}(t)+\theta_2\right)^3}q(t),
	\ee
	\ese
	satisfies
	\be\label{Lip-G-2}
	\sup_{t\in[0, T]}\left\|\left[G'\left(\tilde{w}\right)q\right](t)\right\|_{H^2(\Omega)}\leq L_G\sup_{t\in[0, T]}\left\|q(t)\right\|_{H^2(\Omega)},
	\ee
	and $G'(\tilde{w}(t)): H_o^2(\Omega)\longrightarrow H_o^2(\Omega)$ satisfies
	\be\label{uniformly-continuous-Fre-G}
	\lim_{h\rightarrow0}\sup_{\begin{smallmatrix}0\leq t\leq t+ h\leq T\\ 0\leq\tau\leq 1\end{smallmatrix}}\left\|G'\left(\tilde{w}(t)+\tau\left[\tilde{w}(t+h)-\tilde{w}(t)\right]\right)-G'\left(\tilde{w}(t)\right)\right\|_{\mathcal{B}\left(H_o^2(\Omega)\right)}=0.
	\ee
\end{corollary}
\subsection{Mild and Strict Solutions of Evolution Equations}
\begin{definition}
	Let $\mathfrak{X}$ be a Banach space, $\mathcal{A}: D(\mathcal{A}) \subset \mathfrak{X} \to \mathfrak{X}$ a linear, unbounded operator which  generates a strongly continuous semigroup ($C_0$-semigroup) $\{T(t): t\geq 0\}$. Further, let $T\in(0, \infty)$, $\mathcal{G}\in C([0, T]; \mathfrak{X})$ and $\Phi_0\in \mathfrak{X}$. A function $\Phi$ is called a mild solution of the  inhomogeneous evolution equation
	\be\label{IEE}
	\Phi'(t)=\mathcal{A}\Phi(t)+\mathcal{G}(t), \quad t\in [0, T],\quad \Phi(0)=\Phi_0,
	\ee
	if $\Phi\in C([0, T]; \mathfrak{X})$ is satisfies the integral formulation of \eqref{IEE},
	\be\label{linear solu}
	\Phi(t)=T(t)\Phi_0+\int_0^tT(t-s)\mathcal{G}(s)ds,\quad t\in[0, T].
	\ee
	A function $\Phi$ is said to be a strict solution of \eqref{IEE}, if $\Phi\in C([0, T]; D(\mathcal{A}))\cap C^1([0, T]; \mathfrak{X})$ is a mild solution and satisfies \eqref{IEE}.
\end{definition}

\begin{lemma}\label{IEE-S}
	Let the linear operator $\mathcal{A}$ defined on a Banach space $\mathfrak{X}$ generate the $C_0$-semigroup $\{T(t): t\geq 0\}$, $T\in(0, \infty)$, and $\Phi_0\in D(\mathcal{A})$. If $\mathcal{G}\in C([0, T]; \mathfrak{X})$ and $\Phi$ is a solution of the inhomogeneous evolution equation \eqref{IEE}, then $\Phi$ satisfies the integral formula \eqref{linear solu}.
	
	Assume either that $\mathcal{G}\in C([0, T]; D(\mathcal{A}))$ or that $\mathcal{G}\in C^1([0, T]; \mathfrak{X})$. Then the mild solution $\Phi$ defined by \eqref{linear solu} uniquely solves the inhomogeneous evolution equation \eqref{IEE} on $[0, T]$, and
	\[\Phi\in C([0, T]; D(\mathcal{A}))\cap C^1([0, T]; \mathfrak{X}).\]
\end{lemma}
We refer to Theorem 6.9 in \cite{schnaubelt} for the proof of Lemma \ref{IEE-S}.

\begin{lemma}\label{time-derivative-continuity}
	Let $\mathfrak{X}$ be a Banach space and $\Phi\in C([0, T]; \mathfrak{X})$ be differentiable from the right with right derivative $\Psi\in C\left([0, T]; \mathfrak{X}\right)$. Then $\Phi\in C^1\left([0, T]; \mathfrak{X}\right)$ and $\Phi'=\Psi$.
\end{lemma}
We refer to Lemma 8.9 in \cite{schnaubelt} for the proof of Lemma \ref{time-derivative-continuity}.

\section{Wellposedness of the Linear Elliptic Equation}
\begin{lemma}\label{Chap5REG}
	Let $\Omega \subset \mathbb{R}^n$ be a bounded domain with smooth boundary, $n=1,\ 2$. Let $v\in H^{-1}(\Omega)$ and $w_0\in H^2(\Omega)$ with $\kappa=\displaystyle\inf_{x\in{\Omega}}w_0>0$. Then there exists a constant $C=C(\Omega)>0$ such that for all $w\in B_r(w_0)=\left\{\psi\in H^2(\Omega): \ \left\|\psi-w_0\right\|_{H^2(\Omega)}\leq r\right\}$, $r\in\left(0, \frac{\kappa}{2C}\right)$, the elliptic boundary value problem
	\be\label{Chap5BVP-ellipticfor4thorder}
	\nabla\cdot\left(w^3\nabla \tilde{u}\right)=v\ \text{in}\ \Omega,\quad \tilde{u}=0\ \text{on}\ \partial\Omega,
	\ee
	admits a unique weak solution $\tilde{u}\in H_0^1(\Omega)$, and $\tilde{u}$ satisfies
	\be\label{Chap5reg-u-H1}
	\left\|\tilde{u}\right\|_{H^1(\Omega)}\leq C_o\left\|v\right\|_{H^{-1}(\Omega)}.
	\ee
Here $C_o>0$ is a constant depending only on $\kappa$ and $\Omega$.
\end{lemma}
\begin{proof}
	Since $w_0\in H^2(\Omega)$ and $w\in B_r(w_0)=\left\{\psi\in H^2(\Omega): \ \left\|\psi-w_0\right\|_{H^2(\Omega)}\leq r\right\}$, according to the triangle inequality and the Sobolev embedding theorem, there exists a constant $C=C(\Omega)>0$, such that for all $r\in\left(0, \frac{\kappa}{2C}\right)$, it follows that
	\bse\label{Chap5w-lower-bound3}
	\be\label{Chap5w-lower-bound1-1}
	w=w_0+w-w_0\geq\kappa-\left\|w-w_0\right\|_{L^{\infty}(\Omega)}\geq\kappa-C\left\|w-w_0\right\|_{H^{2}(\Omega)}\geq\kappa-Cr\geq\frac{\kappa}{2},
	\ee
	\be\label{Chap5w-upper-bound1-2}
	\left\|w\right\|_{H^2(\Omega)}\leq\left\|w_0\right\|_{H^2(\Omega)}+\frac{\kappa}{2C}.
	\ee
	\ese
	For all functions that vanish on $\partial\Omega$, from the Poincar\'{e} inequality, we obtain
	\be\label{Chap5P-inequ}
	\left\|\tilde{u}\right\|_{L^2(\Omega)}\leq C\left\|\nabla\tilde{u}\right\|_{L^2(\Omega)}.
	\ee
	It is easy to see that for all $\tilde{u}_1,\ \tilde{u}_2\in H_0^1(\Omega)$
	\begin{align}
		\int_{\Omega}w^3\nabla\tilde{u}_1\nabla\tilde{u}_2dx\leq&\left\|w^3\right\|_{L^\infty(\Omega)}\left\|\nabla\tilde{u}_1\right\|_{L^2(\Omega)}\left\|\nabla\tilde{u}_2\right\|_{L^2(\Omega)}\notag\\
		\leq&\left(\left\|w_0\right\|_{H^2(\Omega)}+\frac{\kappa}{2C}\right)^3\left\|\tilde{u}_1\right\|_{H^1(\Omega)}\left\|\tilde{u}_2\right\|_{H^1(\Omega)}, \label{after23_1}
	\end{align}
	\begin{align}
		\int_{\Omega}w^3\left|\nabla\tilde{u}_1\right|^2dx\geq\frac{\kappa^3}{8}\left\|\nabla\tilde{u}_1\right\|^2_{L^2(\Omega)}\geq\frac{\kappa^3C'}{16}\left\|\tilde{u}_1\right\|^2_{H^1(\Omega)},\label{after23_2}
	\end{align}
	where $C'=\min\left\{1,\ C^{-2}\right\}$. From the Lax-Milgram Theorem (Theorem 1, Section 6.2, \cite{EL}),  the problem \eqref{Chap5BVP-ellipticfor4thorder} has a unique weak solution $\tilde{u}\in H_0^1(\Omega)$, and from \eqref{after23_1}, \eqref{after23_2} this weak solution satisfies \eqref{Chap5reg-u-H1}. 
\end{proof}

\begin{remark}\label{Chap5rem0}
	Define the solution operator $\mathbf{S}_{o}$ by
	\be\label{Chap5solutionop}
	\mathbf{S}_o:\ H^{-1}(\Omega)\times B_r(w_0)\longrightarrow H_0^1(\Omega),\quad (v, w)\longmapsto \tilde{u},\quad  \tilde{u}=\mathbf{S}_o(v, w),
	\ee
	where $\tilde{u}$ denotes the solution of \eqref{Chap5BVP-ellipticfor4thorder}. Then  $\mathbf{S}_o$ is linear with respect to $v$,
	\[\mathbf{S}_o(v_1-v_2, w)=\mathbf{S}_o(v_1, w)-\mathbf{S}_o(v_2, w),\quad v_1,\ v_2\in H^{-1}(\Omega), \ w\in B_r(w_0),\]
	and it satisfies
	\be\label{Chap5Lip-con-sol1}
	\left\|\mathbf{S}_{o}(v_1, w)-\mathbf{S}_{o}(v_2, w)\right\|_{H^1(\Omega)}\leq C_o\left\|v_1-v_2\right\|_{H^{-1}(\Omega)}.
	\ee
	In fact, $\tilde{u}_1=\mathbf{S}_o(v_1, w)$ and $\tilde{u}_2=\mathbf{S}_o(v_2, w)$ denote the solutions of the equations
	\[\nabla\cdot\left(w^3\nabla \tilde{u}_1\right)=v_1\ \text{in}\ \Omega,\quad \tilde{u}_1=0\ \text{on}\ \partial\Omega,\quad \nabla\cdot\left(w^3\nabla \tilde{u}_2\right)=v_2\ \text{in}\ \Omega,\quad \tilde{u}_2=0\ \text{on}\ \partial\Omega,\]
	respectively for each  $w\in B_r(w_0)$ and $v_1$, $v_2\in H^{-1}(\Omega)$. Then $\tilde{u}_1-\tilde{u}_2$ is the solution of
	\[\nabla\cdot\left(w^3\nabla\left[\tilde{u}_1-\tilde{u}_2\right]\right)=v_1-v_2\ \text{in}\ \Omega,\quad \tilde{u}_1-\tilde{u}_2=0\ \text{on}\ \partial\Omega,\]
	for every  $w\in B_r(w_0)$ and $v_1-v_2\in L^2(\Omega)$, and $\tilde{u}_1-\tilde{u}_2=\mathbf{S}_o(v_1-v_2, w)$. By using the estimate \eqref{Chap5reg-u-H1} from Lemma \ref{Chap5REG}, we conclude the assertion \eqref{Chap5Lip-con-sol1}.
	
\end{remark}

\begin{lemma}\label{Chap5LIP-Elliptic}
	$\mathbf{S}_{o}$ satisfies the following Lipschitz continuity estimates:
	\be\label{Chap5Lip-con-sol2}
	\left\|\mathbf{S}_{o}(v, w_1)-\mathbf{S}_{o}(v, w_2)\right\|_{H^1(\Omega)}\leq C_o^*\left\|v\right\|_{H^{-1}(\Omega)}\left\|w_1-w_2\right\|_{H^2(\Omega)}
	\ee
	Here, $C_o^*$ is a constant depending on $C_o$ from Lemma \ref{Chap5REG}, on $\Omega$, $\kappa$ and $\left\|w_0\right\|_{H^2(\Omega)}$.
\end{lemma}
\begin{proof}
	Let $\tilde{u}_1^*=\mathbf{S}_o(v, w_1)$, $\tilde{u}_2^*=\mathbf{S}_o(v, w_2)$ denote the solutions of the equations
	\be\label{Chap5equ1}
	\nabla\cdot\left(w_1^3\nabla \tilde{u}^*_1\right)=v\ \text{in}\ \Omega,\quad \tilde{u}^*_1=0\ \text{on}\ \partial\Omega,
	\ee
	\be\label{Chap5equ2}
	\nabla\cdot\left(w_2^3\nabla \tilde{u}^*_2\right)=v\ \text{in}\ \Omega,\quad \tilde{u}^*_2=0\ \text{on}\ \partial\Omega,
	\ee
	respectively, for given $w_1,\ w_2\in B_r(w_0)$ and $v\in H^{-1}(\Omega)$. Then $\tilde{u}_1^*-\tilde{u}_2^*$ satisfies
	\be\label{Chap5e01}
	\nabla\cdot\left[w_1^3\nabla\left(\tilde{u}^*_1-\tilde{u}^*_2\right)\right]=\nabla\cdot\left[\left(w_2^3-w_1^3\right)\nabla \tilde{u}^*_2\right].
	\ee
	Multiply \eqref{Chap5e01} by $\tilde{u}^*_1-\tilde{u}^*_2$, and integrate over $\Omega$, integrating by parts:
	\be\label{Chap5e02}
	\int_\Omega w_1^3\left|\nabla\left(\tilde{u}^*_1-\tilde{u}^*_2\right)\right|^2dx=\int_{\Omega}\nabla\left(\tilde{u}^*_1-\tilde{u}^*_2\right)\cdot\left[\left(w_2^3-w_1^3\right)\nabla \tilde{u}^*_2\right]dx.
	\ee
	Because $\tilde{u}_2^*$ is a solution of equation \eqref{Chap5equ2}, Lemma \ref{Chap5REG} implies
	\be\label{Chap5e03}
	\left\|\nabla\tilde{u}^*_2\right\|_{L^2(\Omega)}\leq C_o\left\|v\right\|_{H^{-1}(\Omega)}.
	\ee
	From  the algebraic property of $H^2(\Omega)$ and triangle inequality we obtain
	\begin{align}
		\left\|w_2^3-w_1^3\right\|_{H^2(\Omega)}\leq&\left[\left\|w_1\right\|^2_{H^2(\Omega)}+\left\|w_2\right\|^2_{H^2(\Omega)}+\left\|w_1\right\|_{H^2(\Omega)}\left\|w_2\right\|_{H^2(\Omega)}\right]\left\|w_1-w_2\right\|_{H^2(\Omega)}\notag\\
		\leq&3\left(\left\|w_0\right\|_{H^2(\Omega)}+\frac{\kappa}{2C}\right)^2\left\|w_1-w_2\right\|_{H^2(\Omega)}\label{Chap5e04}
	\end{align}
	Using the Sobolev embedding Theorem, the Schwarz inequality, \eqref{Chap5e03} and \eqref{Chap5e04} on the right hand side of \eqref{Chap5e02}, we get
	\begin{align}
		\left|\int_{\Omega}\nabla\left(\tilde{u}^*_1-\tilde{u}^*_2\right)\cdot\left[\left(w_2^3-w_1^3\right)\nabla \tilde{u}^*_2\right]dx\right|\leq&\left\|w_2^3-w_1^3\right\|_{L^\infty(\Omega)}\int_{\Omega}\left|\nabla\left(\tilde{u}^*_1-\tilde{u}^*_2\right)\cdot\nabla \tilde{u}^*_2\right|dx\notag\\
		\leq&C\left\|w_2^3-w_1^3\right\|_{H^2(\Omega)}\left\|\nabla\left(\tilde{u}^*_1-\tilde{u}^*_2\right)\right\|_{L^2(\Omega)}\left\|\nabla \tilde{u}^*_2\right\|_{L^2(\Omega)}\notag\\
		\leq&3CC_o\left(\left\|w_0\right\|_{H^2(\Omega)}+\frac{\kappa}{2C}\right)^2\left\|v\right\|_{H^{-1}(\Omega)}\notag\\
		\cdot&\left\|w_1-w_2\right\|_{H^2(\Omega)}\left\|\nabla\left(\tilde{u}^*_1-\tilde{u}^*_2\right)\right\|_{L^2(\Omega)}\label{Chap5e05}.
	\end{align}
	Because $w_1\in B_r(w_0)$, from estimate \eqref{Chap5w-lower-bound1-1} in the proof of Lemma \ref{Chap5REG}, we have
	\[w_1^3\geq\frac{\kappa^3}{8}.\]
	Hence, the left hand side of \eqref{Chap5e02} satisfies
	\be\label{Chap5e06}
	\int_\Omega w_1^3\left|\nabla\left(\tilde{u}^*_1-\tilde{u}^*_2\right)\right|^2dx\geq\frac{\kappa^3}{8}\left\|\nabla\left(\tilde{u}^*_1-\tilde{u}^*_2\right)\right\|_{L^2(\Omega)}^2.
	\ee
	Therefore, \eqref{Chap5e05} and \eqref{Chap5e06} imply \eqref{Chap5Lip-con-sol2} holds by setting
	\[C_o^*=\frac{24CC_o}{\kappa^3}\left(\left\|w_0\right\|_{H^2(\Omega)}+\frac{\kappa}{2C}\right)^2\sqrt{C^2+1}.\]
\end{proof}

\begin{remark}
	Because of the continuity of the solution operator, we obtain corresponding results in spaces of continuous functions $C\left([0, T]; X\right)$, for appropriate $X$. 
	For example, Lemma \ref{Chap5REG} becomes:\\
	For each $v\in C\left([0, T]; H^{-1}(\Omega)\right)$ and $w\in C\left([0, T]; B_r(w_0)\right)$, there exists a unique solution $\tilde{u}\in C\left([0, T]; H_0^1(\Omega)\right)$ of the linear elliptic equation \eqref{Chap5BVP-ellipticfor4thorder} such that
	\be\label{Chap5time-reg-u-H2}
	\sup_{t\in[0, T]}\left\|\tilde{u}(t)\right\|_{H^1(\Omega)}\leq C_o\sup_{t\in[0, T]}\left\|v(t)\right\|_{H^{-1}(\Omega)}.
	\ee
	
	The solution operator $\mathbf{S}_o$ given by \eqref{Chap5solutionop} becomes
	\bse\label{Chap5solutionop-time}
	\be
	\mathbf{S}_o:\ C\left([0, T]; H^{-1}(\Omega)\times B_r(w_0)\right)\longrightarrow C\left([0, T]; H_0^1(\Omega)\right),\ (v, w)\longmapsto \tilde{u},\  \tilde{u}=\mathbf{S}_o(v, w), 
	\ee
	\be
	\tilde{u}(t)=[\mathbf{S}_o(v, w)](t)=\mathbf{S}_{o}(v(t), w(t)),
	\ee
	\ese
	and $\mathbf{S}_o$ has the similar Lipschitz continuity properties such as
	\begin{align}
		\sup_{t\in[0, T]}\left\|[\mathbf{S}_{o}(v_1, w)](t)-[\mathbf{S}_{o}(v_2, w)](t)\right\|_{H^1(\Omega)}\leq C_o&\sup_{t\in[0, T]}\left\|v_1(t)-v_2(t)\right\|_{H^{-1}(\Omega)}\label{Chap5times-Lip-con-sol1},
	\end{align}
	\begin{align}
		&\sup_{t\in[0, T]}\left\|[\mathbf{S}_{o}(v, w_1)](t)-[\mathbf{S}_{o}(v, w_2)](t)\right\|_{H^1(\Omega)}\notag\\
		\leq&C_o^*\sup_{t\in[0, T]}\left\|v(t)\right\|_{H^{-1}(\Omega)}\sup_{t\in[0, T]}\left\|w_1(t)-w_2(t)\right\|_{H^2(\Omega)}.\label{Chap5time-Lip-con-sol2}
	\end{align}
	Hence
	\begin{align}
		&\sup_{t\in[0, T]}\left\|[\mathbf{S}_{o}(v_1, w_1)](t)-[\mathbf{S}_{o}(v_2, w_2)](t)\right\|_{H^1(\Omega)}\notag\\
		\leq&C_o^*\sup_{t\in[0, T]}\left\|v_2(t)\right\|_{H^{-1}(\Omega)}\sup_{t\in[0, T]}\left\|w_1(t)-w_2(t)\right\|_{H^2(\Omega)}+C_o\sup_{t\in[0, T]}\left\|v_1(t)-v_2(t)\right\|_{H^{-1}(\Omega)}.\label{Chap5time-Lip-con-sol3}
	\end{align}
	Choose $h\in(0, T)$ sufficiently small such that $t+h\in[0, T]$ and $w(t)$, $w(t+h)\in B_r(w_0)$, similarly, we obtain
	\begin{align}
		&\sup_{t\in[0, T]}\left\|\mathbf{S}_{o}(v(t+h), w(t+h))-\mathbf{S}_{o}(v(t), w(t))\right\|_{H^1(\Omega)}\notag\\
		\leq&C_o^*\sup_{t\in[0, T]}\left\|v(t)\right\|_{H^{-1}(\Omega)}\sup_{t\in[0, T]}\left\|w(t+h)-w(t)\right\|_{H^2(\Omega)}+C_o\sup_{t\in[0, T]}\left\|v(t+h)-v(t)\right\|_{H^{-1}(\Omega)}.\label{Chap5time-Lip-con-sol4}
	\end{align}
	Estimates \eqref{Chap5times-Lip-con-sol1} and \eqref{Chap5time-Lip-con-sol2} imply  the Fr\'{e}chet derivative $D_v\mathbf{S}_o(v(t), w(t))$ of  $\mathbf{S}_o(v(t), w(t))$ on $v(t)$ and the Fr\'{e}chet derivative $D_w\mathbf{S}_o(v(t), w(t))$ of  $\mathbf{S}_o(v(t), w(t))$ on $w(t)$ exist and
	\bse\label{Chap5Fre-Sol-v}
	\be\label{Chap5Fre-Sol-v1}
	D_v\mathbf{S}_o(v(t), w(t)): H^{-1}(\Omega)\longrightarrow H_0^1(\Omega),\quad \varphi\longmapsto  [D_v\mathbf{S}_o(v(t), w(t))]\varphi,
	\ee
	\be\label{Chap5Fre-Sol-v2}
	[D_v\mathbf{S}_o(v(t), w(t))]\varphi=\lim_{\lambda\rightarrow 0}\frac{1}{\lambda}\left[\mathbf{S}_o(v(t)+\lambda\varphi, w(t))-\mathbf{S}_o(v(t), w(t))\right],
	\ee
	\ese
	\bse\label{Chap5Fre-Sol-w}
	\be\label{Chap5Fre-Sol-w1}
	D_w\mathbf{S}_o(v(t), w(t)): H^{2}(\Omega)\longrightarrow H_0^1(\Omega),\quad \psi\longmapsto  [D_w\mathbf{S}_o(v(t), w(t))]\psi,
	\ee
	\be\label{Chap5Fre-Sol-w2}
	[D_w\mathbf{S}_o(v(t), w(t))]\psi=\lim_{\lambda\rightarrow 0}\frac{1}{\lambda}\left[\mathbf{S}_o(v(t), w(t)+\lambda\psi)-\mathbf{S}_o(v(t), w(t))\right].
	\ee
	\ese
	According to inequalities \eqref{Chap5times-Lip-con-sol1}, \eqref{Chap5time-Lip-con-sol2}, it is easy to obtain that functions $D_v\mathbf{S}_o(v(t), w(t))\phi$ and \\$D_v\mathbf{S}_o(v(t), w(t))\psi$ are uniformly continuous with respect to $\phi$ and $\psi$ respectively.
\end{remark}

\begin{lemma}\label{Chap5Fre-Lip-continuity}
	Let $v\in C\left([0, T]; H^{-1}(\Omega)\right)$, $w\in C\left([0, T]; B_r(w_0)\right)$. Then
	\be\label{Chap5Fre-Lip-continuity-v}
	\sup_{t\in[0, T]}\left\|[D_v\mathbf{S}_o\left(v(t), w(t)\right)]\varphi_1-[D_v\mathbf{S}_o(v(t), w(t))]\varphi_2\right\|_{H^1(\Omega)}\leq C_o\left\|\varphi_1-\varphi_2\right\|_{H^{-1}(\Omega)}
	\ee
	holds for all $\varphi_1,\ \varphi_2\in H^{-1}(\Omega)$.
	\begin{align}
		&\sup_{t\in[0, T]}\left\|[D_w\mathbf{S}_o\left(v(t), w(t)\right)]\psi_1-[D_w\mathbf{S}_o(v(t), w(t))]\psi_2\right\|_{H^1(\Omega)}\notag\\
		\leq&C_o^*\sup_{t\in[0, T]}\left\|v(t)\right\|_{H^{-1}(\Omega)}\left\|\psi_1-\psi_2\right\|_{H^{2}(\Omega)}\label{Chap5Fre-Lip-continuity-w}
	\end{align}
	holds for all $\psi_1,\ \psi_2\in H^{2}(\Omega)$.
\end{lemma}
\begin{proof}
	From the definitions \eqref{Chap5Fre-Sol-v}, \eqref{Chap5Fre-Sol-w} of $D_v\mathbf{S}_o\left(v(t), w(t)\right)$ and $D_w\mathbf{S}_o\left(v(t), w(t)\right)$ and estimates \eqref{Chap5times-Lip-con-sol1}, \eqref{Chap5time-Lip-con-sol2}, it is easy to see that
	\begin{align}
		&\sup_{t\in[0, T]}\left\|D_v\mathbf{S}_o\left(v(t), w(t)\right)\varphi_1-D_v\mathbf{S}_o(v(t), w(t))\varphi_2\right\|_{H^1(\Omega)}\notag\\
		=&\lim_{\lambda\rightarrow 0}\frac{1}{\lambda}\sup_{t\in[0, T]}\left\|\mathbf{S}_o(v(t)+\lambda\varphi_1, w(t))-\mathbf{S}_o(v(t)+\lambda\varphi_2, w(t))\right\|_{H^1(\Omega)}\notag\\
		\leq&\lim_{\lambda\rightarrow 0}\frac{1}{\lambda}C_o\left\|\lambda\left(\varphi_1-\varphi_2\right)\right\|_{H^{-1}(\Omega)}\notag\\
		=&C_o\left\|\varphi_1-\varphi_2\right\|_{H^{-1}(\Omega)}\notag,
	\end{align}
	\begin{align}
		&\sup_{t\in[0, T]}\left\|D_w\mathbf{S}_o\left(v(t), w(t)\right)\psi_1-D_w\mathbf{S}_o(v(t), w(t))\psi_2\right\|_{H^1(\Omega)}\notag\\
		=&\lim_{\lambda\rightarrow 0}\frac{1}{\lambda}\sup_{t\in[0, T]}\left\|\mathbf{S}_o(v(t), w(t)+\lambda\psi_1)-\mathbf{S}_o(v(t), w(t)+\lambda\psi_2)\right\|_{H^1(\Omega)}\notag\\
		\leq&\lim_{\lambda\rightarrow 0}\frac{1}{\lambda}C_o^*\sup_{t\in[0, T]}\left\|v(t)\right\|_{H^{-1}(\Omega)}\left\|\lambda\left(\psi_1-\psi_2\right)\right\|_{H^{-1}(\Omega)}\notag\\
		=&C_o^*\sup_{t\in[0, T]}\left\|v(t)\right\|_{H^{-1}(\Omega)}\left\|\psi_1-\psi_2\right\|_{H^{-1}(\Omega)}\notag.
	\end{align}
\end{proof}

\begin{remark}
	Lemma \ref{Chap5Fre-Lip-continuity} implies
	\bse\label{Chap5Bound-of-Fre-D}
	\be\label{Chap5Bound-of-Fre-D1}
	\sup_{t\in[0, T]}\left\|D_v\mathbf{S}_o\left(v(t), w(t)\right)\right\|_{\mathcal{B}\left(H^{-1}(\Omega),\ H_0^1(\Omega)\right)}\leq C_o,
	\ee
	\be\label{Chap5Bound-of-Fre-D2}
	\sup_{t\in[0, T]}\left\|D_w\mathbf{S}_o\left(v(t), w(t)\right)\right\|_{\mathcal{B}\left(H^2(\Omega),\ H_0^1(\Omega)\right)}\leq C_o^*\sup_{t\in[0, T]}\left\|v(t)\right\|_{H^{-1}(\Omega)}.
	\ee
	\ese
	Now we are going to show that $D_v\mathbf{S}_o(v(t), w(t))$ and $D_v\mathbf{S}_o(v(t), w(t))$ are uniformly continuous on the compact sets \[\left\{v(t)+\tau\left[v(t+h)-v(t)\right]\in H^{-1}(\Omega):\ \tau\in[0, 1],\ t, t+h\in[0, T]\right\},\]
	\[\left\{w(t)+\tau\left[w(t+h)-w(t)\right]\in B_r(w_0):\ \tau\in[0, 1],\ t, t+h\in[0, T]\right\}\]
	respectively, which is Lemma \ref{Chap5Fre-uniformly-continuity}.
\end{remark}

\begin{lemma}\label{Chap5Fre-uniformly-continuity}
	Let $v\in C\left([0, T]; H^{-1}(\Omega)\right)$, $w\in C\left([0, T]; B_r(w_0)\right)$ and choose $h\in(0, T)$ sufficiently small such that $w(t+h)\in B_r(w_0)$ for all $t\in[0, T]$. Denote $v_h(t)=v(t+h)-v(t)$ and $w_h(t)=w(t+h)-w(t)$. Then
	\begin{align}
		&\sup_{\begin{smallmatrix} t, t+h\in[0, T]\\ \tau\in[0, 1]\end{smallmatrix}}\left\|D_v\mathbf{S}_o\left(v(t)+\tau v_h(t), w(t)+\tau w_h(t)\right)-D_v\mathbf{S}_o(v(t), w(t))\right\|_{\mathcal{B}\left(H^{-1}(\Omega),\ H_0^1(\Omega)\right)}\notag\\
		&:=\alpha_1(h)\rightarrow 0,\ \text{as}\ h\rightarrow 0\label{Chap5Fre-uniformly-continuity-v},
	\end{align}
	\begin{align}
		&\sup_{\begin{smallmatrix} t, t+h\in[0, T]\\ \tau\in[0, 1]\end{smallmatrix}}\left\|D_w\mathbf{S}_o\left(v(t)+\tau v_h(t), w(t)+\tau w_h(t)\right)-D_w\mathbf{S}_o(v(t), w(t))\right\|_{\mathcal{B}\left(H^{2}(\Omega),\ H_0^1(\Omega)\right)}\notag\\
		&:=\alpha_2(h)\rightarrow 0,\ \text{as}\ h\rightarrow 0\label{Chap5Fre-uniformly-continuity-w}.
	\end{align}
\end{lemma}
\begin{proof}
	We denote
	\[\tilde{u}_1=\mathbf{S}_o\left(v(t)+\tau v_h(t)+\lambda\varphi,  w(t)+\tau w_h(t)\right),\ \tilde{u}_2=\mathbf{S}_o\left(v(t)+\tau v_h(t),  w(t)+\tau w_h(t)\right),\]
	\[\tilde{u}_1^*=\mathbf{S}_o\left(v(t)+\lambda\varphi, w(t)\right),\quad \tilde{u}_2^*=\mathbf{S}_o\left(v(t), w(t)\right),\]
	and then $\tilde{u}_1-\tilde{u}_2$ and $\tilde{u}_1^*-\tilde{u}_2^*$ satisfy
	\be\label{Chap5e07}
	\lambda\varphi=\nabla\cdot\left\{( w(t)+\tau w_h(t))^3\nabla\left(\tilde{u}_1-\tilde{u}_2\right)\right\},\quad \tilde{u}_1-\tilde{u}_2=0,\ \text{on}\ \partial\Omega,
	\ee
	\be\label{Chap5e08}
	\lambda\varphi=\nabla\cdot\left\{[w(t)]^3\nabla\left(\tilde{u}_1^*-\tilde{u}_2^*\right)\right\},\quad \tilde{u}_1^*-\tilde{u}_2^*=0,\ \text{on}\ \partial\Omega,
	\ee
	respectively. So $w(t)+\tau w_h(t)\in B_r(w_0)$ and Lemma \ref{Chap5REG} imply $\tilde{u}_1^*-\tilde{u}_2^*$ satisfies
	\[\left\|\nabla\left(\tilde{u}^*_1-\tilde{u}^*_2\right)\right\|_{L^2(\Omega)}\leq C_o\left\|\lambda\varphi\right\|_{H^{-1}(\Omega)},\]
	Now, estimates \eqref{Chap5e07} and \eqref{Chap5e08} imply $\tilde{u}_1-\tilde{u}_2-\left(\tilde{u}_1^*-\tilde{u}_2^*\right)$ satisfies
	\begin{align}
		&\nabla\cdot\left\{( w(t)+\tau w_h(t))^3\nabla\left[\tilde{u}_1-\tilde{u}_2-\left(\tilde{u}_1^*-\tilde{u}_2^*\right)\right]\right\}\notag\\
		=&-\nabla\cdot\left\{\left(( w(t)+\tau w_h(t))^3-[w(t)]^3\right)\nabla\left(\tilde{u}_1^*-\tilde{u}_2^*\right)\right\}\label{Chap5e09}
	\end{align}
	Multiply \eqref{Chap5e09} by  $\tilde{u}_1-\tilde{u}_2-\left(\tilde{u}_1^*-\tilde{u}_2^*\right)$, and integrate over $\Omega$, integrating by parts:
	\begin{align}
		&\int_\Omega\left( w(t)+\tau w_h(t)\right)^3\left|\nabla\left[\tilde{u}_1-\tilde{u}_2-\left(\tilde{u}_1^*-\tilde{u}_2^*\right)\right]\right|^2dx\notag\\
		=&-\int_\Omega\left(( w(t)+\tau w_h(t))^3-[w(t)]^3\right)\nabla\left(\tilde{u}_1^*-\tilde{u}_2^*\right)\cdot\nabla\left[\tilde{u}_1-\tilde{u}_2-\left(\tilde{u}_1^*-\tilde{u}_2^*\right)\right]dx\label{Chap5e10}
	\end{align}
	As $w(t)+\tau w_h(t)\in B_r(w_0)$, similar to
\eqref{after23_2}, \eqref{Chap5e04}, \eqref{Chap5e05} and \eqref{Chap5e06},  we obtain
	\begin{align}
		\int_\Omega( w(t)+\tau w_h(t))^3\left|\nabla\left[\tilde{u}_1-\tilde{u}_2-\left(\tilde{u}_1^*-\tilde{u}_2^*\right)\right]\right|^2dx\geq\frac{\kappa^3}{8}\left\|\nabla\left[\tilde{u}_1-\tilde{u}_2-\left(\tilde{u}_1^*-\tilde{u}_2^*\right)\right]\right\|_{L^2(\Omega)}^2,\notag
	\end{align}
	\begin{align}
		&\left|-\int_\Omega\left(\left(w(t)+\tau w_h(t)\right)^3-[w(t)]^3\right)\nabla\left(\tilde{u}_1^*-\tilde{u}_2^*\right)\cdot\nabla\left[\tilde{u}_1-\tilde{u}_2-\left(\tilde{u}_1^*-\tilde{u}_2^*\right)\right]dx\right|\notag\\
		\leq&C\left\|\left(w(t)+\tau w_h(t)\right)^3-[w(t)]^3\right\|_{H^2(\Omega)}\left\|\nabla\left(\tilde{u}^*_1-\tilde{u}^*_2\right)\right\|_{L^2(\Omega)}\left\|\nabla \left[\tilde{u}_1-\tilde{u}_2-\left(\tilde{u}_1^*-\tilde{u}_2^*\right)\right]\right\|_{L^2(\Omega)}\notag\\
		\leq&3CC_o\left(\left\|w_0\right\|_{H^2(\Omega)}+\frac{\kappa}{2C}\right)^2\left\|\lambda\varphi\right\|_{H^{-1}(\Omega)}\left\|w_h(t)\right\|_{H^2(\Omega)}\left\|\nabla\left[\tilde{u}_1-\tilde{u}_2-\left(\tilde{u}_1^*-\tilde{u}_2^*\right)\right]\right\|_{L^2(\Omega)}.\notag
	\end{align}
	Recall $w_h(t)=w(t+h)-w(t)$. Thus
	\be\label{Chap5e12}
	\left\|\tilde{u}_1-\tilde{u}_2-\left(\tilde{u}_1^*-\tilde{u}_2^*\right)\right\|_{H^1(\Omega)}\leq C_o^*\left\|\lambda\varphi\right\|_{H^{-1}(\Omega)}\left\|w(t+h)-w(t)\right\|_{H^2(\Omega)}.
	\ee
	According to the definitions \eqref{Chap5Fre-Sol-v} of Fr\'{e}chet derivative $D_v\mathbf{S}_o(v(t), w(t))$,
	\begin{align}
		\alpha_1(h)=&\sup_{\begin{smallmatrix} t, t+h\in[0, T]\\ \tau\in[0, 1],\ \varphi\in H^{-1}(\Omega)\end{smallmatrix}}\frac{\left\|\lim_{\lambda\rightarrow0}\frac{1}{\lambda}\left[\tilde{u}_1-\tilde{u}_2-\left(\tilde{u}_1^*-\tilde{u}_2^*\right)\right]\right\|_{H^1(\Omega)}}{\left\|\varphi\right\|_{H^{-1}(\Omega)}}\notag\\
		\leq&\sup_{0\leq t<t+h\leq T}C_o^*\left\|w(t+h)-w(t)\right\|_{H^2(\Omega)}\rightarrow0,\ \text{as}\ h\rightarrow0 \notag,
	\end{align}
	because $w\in C\left([0, T]; B_r(w_0)\right)$ and $w(t)$ is uniformly continuous with respect to $t\in[0, T]$.
	
	Hence we conclude the assertion \eqref{Chap5Fre-uniformly-continuity-v}. We next show that the assertion \eqref{Chap5Fre-uniformly-continuity-w} holds.
	
	Recall $v_h(t)=v(t+h)-v(t)$, $w_h(t)=w(t+h)-w(t)$. Choose $\psi\in H^2(\Omega)$ and small $\lambda$ such that $w(t)+\tau w_h(t)$, $w(t)+\tau w_h(t)+\lambda\psi\in B_r(w_0)$. Write
	\[\tilde{u}_3=\mathbf{S}_o\left(v(t)+\tau v_h(t), w(t)+\tau w_h(t)+\lambda\psi\right),\quad \tilde{u}_4=\mathbf{S}_o\left(v(t)+\tau v_h(t), w(t)+\tau w_h(t)\right),\]
	\[\tilde{u}_3^*=\mathbf{S}_o\left(v(t), w(t)+\lambda\psi\right),\quad \tilde{u}_4^*=\mathbf{S}_o\left(v(t), w(t)\right).\]
	Then $\tilde{u}_3-\tilde{u}_4$ and $\tilde{u}_3^*-\tilde{u}_4^*$ satisfy
	\begin{align}\label{Chap5e13}
		&\nabla\cdot\left\{\left(w(t)+\tau w_h(t)+\lambda\psi\right)^3\nabla\left(\tilde{u}_3-\tilde{u}_4\right)\right\}\notag\\
		=&-\nabla\cdot\left\{\left([w(t)+\tau w_h(t)+\lambda\psi]^3-\left(w(t)+\tau w_h(t)\right)^3\right)\nabla\tilde{u}_4\right\}\notag,\\
		&\tilde{u}_3-\tilde{u}_4=0,\ \text{on}\ \partial\Omega,
	\end{align}
	\begin{align}
		\nabla\cdot\left\{[w(t)+\lambda\psi]^3\nabla\left(\tilde{u}^*_3-\tilde{u}^*_4\right)\right\}&=-\nabla\cdot\left\{\left([w(t)+\lambda\psi]^3-[w(t)]^3\right)\nabla\tilde{u}_4^*\right\},\notag\\ \tilde{u}_3^*-\tilde{u}_4^*&=0,\ \text{on}\ \partial\Omega,\label{Chap5e14}
	\end{align}
	respectively. Hence, $\tilde{u}_3-\tilde{u}_4-\left(\tilde{u}_3^*-\tilde{u}_4^*\right)$ satisfies
	\begin{align}
		&\nabla\cdot\left\{\left(w(t)+\tau w_h(t)+\lambda\psi\right)^3\nabla\left[\tilde{u}_3-\tilde{u}_4-\left(\tilde{u}_3^*-\tilde{u}_4^*\right)\right]\right\}\notag\\
		=&-\nabla\cdot\left\{\left\{[w(t)+\tau w_h(t)+\lambda\psi]^3-[w(t)+\tau w_h(t)]^3-\left([w(t)+\lambda\psi]^3-[w(t)]^3\right)\right\}\nabla\tilde{u}_4^*\right\}\notag\\
		&-\nabla\cdot\left\{\left([w(t)+\tau w_h(t)+\lambda\psi]^3-[w(t)+\tau w_h(t)]^3\right)\nabla\left(\tilde{u}_4-\tilde{u}_4^*\right)\right\}\notag\\
		&-\nabla\cdot\left\{\left([w(t)+\tau w_h(t)+\lambda\psi]^3-[w(t)+\lambda\psi]^3\right)\nabla\left(\tilde{u}_3^*-\tilde{u}_4^*\right)\right\}\label{Chap5e15}.
	\end{align}
	Multiply \eqref{Chap5e15} by $\tilde{u}_3-\tilde{u}_4-\left(\tilde{u}_3^*-\tilde{u}_4^*\right)$, and integrate over $\Omega$. Integrating by parts, as $w(t)+\tau w_h(t)+\lambda\psi\in B_r(w_0)$, the left hand side of \eqref{Chap5e15} becomes
	\begin{align}
		&\int_\Omega\left(w(t)+\tau w_h(t)+\lambda\psi\right)^3\left|\nabla\left[\tilde{u}_3-\tilde{u}_4-\left(\tilde{u}_3^*-\tilde{u}_4^*\right)\right]\right|^2dx\notag\\
		\geq&\frac{\kappa^3}{8}\left\|\nabla\left[\tilde{u}_3-\tilde{u}_4-\left(\tilde{u}_3^*-\tilde{u}_4^*\right)\right]\right\|_{L^2(\Omega)}^2.\label{Chap5e16}
	\end{align}
	Let
	\[Q_1=[w(t)+\tau w_h(t)+\lambda\psi]^3-[w(t)+\tau w_h(t)]^3-\left([w(t)+\lambda\psi]^3-[w(t)]^3\right)\]
	\[Q_2=[w(t)+\tau w_h(t)+\lambda\psi]^3-[w(t)+\tau w_h(t)]^3.\]
	\[Q_3=[w(t)+\tau w_h(t)+\lambda\psi]^3-[w(t)+\lambda\psi]^3.\]
	Then the first, second, respectively third terms of right hand side of \eqref{Chap5e15}  become
	\be\label{Chap5e17}
	RHS_1=-\int_\Omega Q_1\nabla\tilde{u}_4^*\cdot\nabla\left[\tilde{u}_3-\tilde{u}_4-\left(\tilde{u}_3^*-\tilde{u}_4^*\right)\right]dx.
	\ee
	\be\label{Chap5e20}
	RHS_2=-\int_\Omega Q_2\nabla\left(\tilde{u}_4-\tilde{u}_4^*\right)\cdot\nabla\left[\tilde{u}_3-\tilde{u}_4-\left(\tilde{u}_3^*-\tilde{u}_4^*\right)\right]dx.
	\ee
	\be\label{Chap5e23}
	RHS_3=-\int_\Omega Q_3\nabla\left(\tilde{u}_3^*-\tilde{u}_4^*\right)\cdot\nabla\left[\tilde{u}_3-\tilde{u}_4-\left(\tilde{u}_3^*-\tilde{u}_4^*\right)\right]dx.
	\ee
	Since $\tilde{u}_4^*$ is a solution of equation
	\[\nabla\cdot\left([w(t)]^3\nabla \tilde{u}^*_4\right)=v(t)\ \text{in}\ \Omega,\quad \tilde{u}^*_4=0\ \text{on}\ \partial\Omega,\]
	Lemma \ref{Chap5REG} implies
	\be\label{Chap5RHS1-1}
	\left\|\nabla\tilde{u}^*_4\right\|_{L^2(\Omega)}\leq C_o\left\|v(t)\right\|_{H^{-1}(\Omega)}.
	\ee
	Because $w(t)$, $w(t)+\lambda\psi$, $w(t)+\tau w_h(t)$, $w(t)+\tau w_h(t)+\lambda\psi\in B_r(w_0)$, then
	\[\left\|6w(t)+3\tau w_h(t)+3\lambda\psi\right\|_{H^2(\Omega)}\leq C_a<\infty\]
	for small $\lambda$ and $\tau\in[0, 1]$, where $C_a$ is a constant depending on $w_0$ and $\Omega$. Hence
	\begin{align}
		\left\|Q_1\right\|_{L^\infty(\Omega)}\leq&C\left\|Q_1\right\|_{H^2(\Omega)}\notag\\
		\leq&C\tau \left\|w_h(t)\right\|_{H^2(\Omega)}\left\|6w(t)+3\tau w_h(t)+3\lambda\psi\right\|_{H^2(\Omega)}\lambda\left\|\psi\right\|_{H^2(\Omega)}\notag\\
		\leq&CC_a|\lambda|\left\|w_h(t)\right\|_{H^2(\Omega)}\left\|\psi\right\|_{H^2(\Omega)}.\label{Chap5RHS1-2}
	\end{align}
	Here $C=C(\Omega)>0$ is a constant.  Thus, using \eqref{Chap5RHS1-1}, \eqref{Chap5RHS1-2} and setting $D_1=C_oCC_a$, we obtain
	\begin{align}
		\left|RHS_1\right|\leq& \left\|Q_1\right\|_{L^\infty(\Omega)}\left\|\nabla\tilde{u}_4^*\right\|_{L^2(\Omega)}\left\|\nabla\left[\tilde{u}_3-\tilde{u}_4-\left(\tilde{u}_3^*-\tilde{u}_4^*\right)\right]\right\|_{L^2(\Omega)}\notag\\
		\leq& D_1|\lambda|\left\|\psi\right\|_{H^2(\Omega)}\left\|w_h(t)\right\|_{H^2(\Omega)}\left\|v(t)\right\|_{H^{-1}(\Omega)}\left\|\nabla\left[\tilde{u}_3-\tilde{u}_4-\left(\tilde{u}_3^*-\tilde{u}_4^*\right)\right]\right\|_{L^2(\Omega)}.\label{Chap5e19}
	\end{align}
	Estimates \eqref{Chap5time-Lip-con-sol3} and \eqref{Chap5time-Lip-con-sol4} imply $\tilde{u}_4-\tilde{u}_4^*$ satisfies
	\be\label{Chap5RHS2-1}
	\left\|\nabla\left(\tilde{u}_4-\tilde{u}_4^*\right)\right\|_{L^2(\Omega)}\leq C_o^*\left\|v(t)\right\|_{H^{-1}(\Omega)}\left\|w_h(t)\right\|_{H^2(\Omega)}+C_o\left\|v_h(t)\right\|_{H^{-1}(\Omega)}.
	\ee
	Because $w(t)+\tau w_h(t)$, $w(t)+\tau w_h(t)+\lambda\psi\in B_r(w_0)$, then
	\[\left\|3\left(w(t)+\tau w_h(t)\right)^2+3\lambda\psi\left(w(t)+\tau w_h(t)\right)+\lambda^2\psi^2\right\|_{H^2(\Omega)}\leq C_b<\infty\]
	for small $\lambda$ and $\tau\in[0, 1]$, where $C_b$ is a constant depending on $w_0$ and $\Omega$. Hence
	\be\label{Chap5RHS2-2}
	\left\|Q_2\right\|_{L^\infty(\Omega)}\leq C\left\|Q_2\right\|_{H^2(\Omega)}\leq CC_b|\lambda|\left\|\psi\right\|_{H^2(\Omega)}.
	\ee
	Here $C=C(\Omega)>0$ is a constant.  Thus, using \eqref{Chap5RHS2-1}, \eqref{Chap5RHS2-2} and setting\\
	$D_2=C_o^*CC_b+CC_oC_b$, we obtain
	\begin{align}
		\left|RHS_2\right|\leq& \left\|Q_2\right\|_{L^\infty(\Omega)}\left\|\nabla\left(\tilde{u}_4-\tilde{u}_4^*\right)\right\|_{L^2(\Omega)}\left\|\nabla\left[\tilde{u}_3-\tilde{u}_4-\left(\tilde{u}_3^*-\tilde{u}_4^*\right)\right]\right\|_{L^2(\Omega)}\notag\\
		\leq& D_2|\lambda|\left\|\psi\right\|_{H^2(\Omega)}\left\|v(t)\right\|_{H^{-1}(\Omega)}\left\|w_h(t)\right\|_{H^2(\Omega)}\left\|\nabla\left[\tilde{u}_3-\tilde{u}_4-\left(\tilde{u}_3^*-\tilde{u}_4^*\right)\right]\right\|_{L^2(\Omega)}\notag\\
		+&D_2|\lambda|\left\|\psi\right\|_{H^2(\Omega)}\left\|v_h(t)\right\|_{H^{-1}(\Omega)}\left\|\nabla\left[\tilde{u}_3-\tilde{u}_4-\left(\tilde{u}_3^*-\tilde{u}_4^*\right)\right]\right\|_{L^2(\Omega)}\label{Chap5e22}.
	\end{align}
	Estimates \eqref{Chap5time-Lip-con-sol3} and \eqref{Chap5time-Lip-con-sol4} imply $\tilde{u}_3^*-\tilde{u}_4^*$ satisfies
	\be\label{Chap5e24}
	\left\|\nabla\left(\tilde{u}_3^*-\tilde{u}_4^*\right)\right\|_{L^2(\Omega)}\leq C_o^*|\lambda|\left\|\psi\right\|_{H^2(\Omega)}\left\|v(t)\right\|_{H^{-1}(\Omega)}.
	\ee
	Because $w(t)+\tau w_h(t)+\lambda\psi$, $w(t)+\lambda\psi\in B_r(w_0)$, then
	\[\left\|3\left(w(t)+\lambda\psi\right)^2+3\tau w_h(t)\left(w(t)+\lambda\psi\right)+\tau^2w_h^2(t)\right\|_{H^2(\Omega)}\leq C_e<\infty\]
	for small $\lambda$ and $\tau\in[0, 1]$, where $C_e$ is a constant depending on $w_0$ and $\Omega$. Hence
	\be\label{Chap5e24-1}
	\left\|Q_3\right\|_{L^\infty(\Omega)}\leq C\left\|Q_3\right\|_{H^2(\Omega)}\leq CC_e\left\|w_h(t)\right\|_{H^2(\Omega)}.
	\ee
	Here $C=C(\Omega)>0$ is a constant.  Thus, using \eqref{Chap5e24}, \eqref{Chap5e24-1} and setting $D_3=C_o^*CC_e$, we obtain
	\begin{align}
		\left|RHS_3\right|\leq& \left\|Q_3\right\|_{L^\infty(\Omega)}\left\|\nabla\left(\tilde{u}_3^*-\tilde{u}_4^*\right)\right\|_{L^2(\Omega)}\left\|\nabla\left[\tilde{u}_3-\tilde{u}_4-\left(\tilde{u}_3^*-\tilde{u}_4^*\right)\right]\right\|_{L^2(\Omega)}\notag\\
		\leq& D_3|\lambda|\left\|\psi\right\|_{H^2(\Omega)}\left\|v(t)\right\|_{H^{-1}(\Omega)}\left\|w_h(t)\right\|_{H^2(\Omega)}\left\|\nabla\left[\tilde{u}_3-\tilde{u}_4-\left(\tilde{u}_3^*-\tilde{u}_4^*\right)\right]\right\|_{L^2(\Omega)}\label{Chap5e25}.
	\end{align}
	Combining \eqref{Chap5e16}, \eqref{Chap5e19}, \eqref{Chap5e22} with \eqref{Chap5e25} gives
	\begin{align}
		\left\|\nabla\left[\tilde{u}_3-\tilde{u}_4-\left(\tilde{u}_3^*-\tilde{u}_4^*\right)\right]\right\|_{L^2(\Omega)}\leq&\left[\frac{8(D_1+D_2+D_3)}{\kappa^3}\right]|\lambda|\left\|\psi\right\|_{H^2(\Omega)}\left\|w_h(t)\right\|_{H^2(\Omega)}\left\|v(t)\right\|_{H^{-1}(\Omega)}\notag\\
		+&\frac{8D_2}{\kappa^3}|\lambda|\left\|\psi\right\|_{H^2(\Omega)}\left\|v_h(t)\right\|_{H^{-1}(\Omega)}\label{Chap5e26}.
	\end{align}
	Because of the Poincar\'{e} inequality
	\[\left\|\tilde{u}_3-\tilde{u}_4-\left(\tilde{u}_3^*-\tilde{u}_4^*\right)\right\|_{L^2(\Omega)}\leq C\left\|\nabla\left[\tilde{u}_3-\tilde{u}_4-\left(\tilde{u}_3^*-\tilde{u}_4^*\right)\right]\right\|_{L^2(\Omega)}.\]
	Setting $D_4=\frac{8(D_1+D_2+D_3)}{\kappa^3}+ \frac{8D_2}{\kappa^3}\sqrt{C^2+1}$, we obtain
	\begin{align}
		\left\|\tilde{u}_3-\tilde{u}_4-\left(\tilde{u}_3^*-\tilde{u}_4^*\right)\right\|_{H^1(\Omega)}\leq& D_4|\lambda|\left\|\psi\right\|_{H^2(\Omega)}\left\|w_h(t)\right\|_{H^2(\Omega)}\left\|v(t)\right\|_{H^{-1}(\Omega)}\notag\\
		+&D_4|\lambda|\left\|\psi\right\|_{H^2(\Omega)}\left\|v_h(t)\right\|_{H^{-1}(\Omega)}.\label{Chap5e27}
	\end{align}
	Recall $v_h(t)=v(t+h)-v(t)$ and $w_h(t)=w(t+h)-w(t)$. According to the definitions \eqref{Chap5Fre-Sol-v} for the Fr\'{e}chet derivative $D_w\mathbf{S}_o(v(t), w(t))$,
	\begin{align}
		\alpha_2(h)=&\sup_{\begin{smallmatrix} t, t+h\in[0, T]\\ \tau\in[0, 1],\ \psi\in H^{2}(\Omega)\end{smallmatrix}}\frac{\left\|\lim_{\lambda\rightarrow0}\frac{1}{\lambda}\left[\tilde{u}_3-\tilde{u}_4-\left(\tilde{u}_3^*-\tilde{u}_4^*\right)\right]\right\|_{H^1(\Omega)}}{\left\|\psi\right\|_{H^{2}(\Omega)}}\notag\\
		\leq&\sup_{0\leq t<t+h\leq T}D_4\left(\left\|v(t)\right\|_{H^{-1}(\Omega)}\left\|w(t+h)-w(t)\right\|_{H^2(\Omega)}+\left\|v(t+h)-v(t)\right\|_{H^{-1}(\Omega)}\right)\notag\\
		\rightarrow&0,\ \text{as}\ h\rightarrow0 \notag,
	\end{align}
	because $v\in C\left([0, T]; H^{-1}(\Omega)\right)$, $w\in C\left([0, T]; B_r(w_0)\right)$, $\left\|v(t)\right\|_{H^{-1}(\Omega)}<\infty$, $v(t)$ and $w(t)$ are uniformly continuous with respect to $t\in[0, T]$ respectively.
\end{proof}
\section{Wellposedness of the Nonlinear Coupled System}\label{Chap54th-order problem}

\subsection*{Abstract Formulation of the Nonlinear Coupled System}
Recall that $\beta_F$, $\beta_p$, $\theta_1$ and $\theta_2$ are given positive constants, $\Omega\subset\mathbb{R}^{n}$ is an open and bounded subspace with smooth boundary $\partial\Omega$, $n=1,\ 2$.

Let $v_0\in H_o^2(\Omega)$ and $w_0\in H^4_{\theta_2}(\Omega)$ be  given functions such that the elliptic equation for $\tilde{u}_0$, $\nabla\cdot\left(w_0^3\nabla \tilde{u}_0\right)=v_0$, has a unique solution $\tilde{u}_0\in H_0^1(\Omega)$. Based on the definition \eqref{Chap5solutionop} of solution operator $\mathbf{S}_o$, $\tilde{u}_0=\mathbf{S}_o(v_0, w_0)$.

Restrict $r\in\left(0, \frac{\kappa}{2C}\right)$. Here, $C=C(\Omega)>0$ is a constant, $\kappa=\displaystyle\inf_{x\in{\Omega}}w_0(x)>0$. Set ${u}_0=\tilde{u}_0+\theta_1\in H^1(\Omega)$, $v_0=\tilde{v}_0\in H_o^2(\Omega)$ and $\tilde{w}_0=w_0-\theta_2\in H^4_o(\Omega)$.  Take $T\in(0, \infty)$ to be specified below. We introduce a state $\mathbf{a}=\left(a_1,\ a_2\right)$, a state space $\mathfrak{X}$ defined by 
\be\label{state-space}
\mathfrak{X}=L^2(\Omega)\times H_o^2(\Omega)
\ee
endowed with the norm $\|\cdot\|_{\mathfrak{X}}=\left\|\cdot\right\|_{L^2(\Omega)\times H^2(\Omega)}$ and the scalar product
\[\langle \mathbf{a}, \mathbf{b}\rangle_{\mathfrak{X}}=\int_\Omega a_1\cdot\overline{b_1}+\nabla a_2\cdot\nabla\overline{b_2}+\Delta a_2\cdot\Delta\overline{b_2}dx,\quad \mathbf{a}=\left(a_1, a_2\right)\in\mathfrak{X}, \ \mathbf{b}=\left(b_1, b_2\right)\in\mathfrak{X}.\] 
We then define a operator $\mathbb{A}$ by 
\begin{align}
	 D(\mathbb{A}):=&\bigg\{\phi\in H_o^2(\Omega): \exists\ f\in L^2(\Omega),\ \forall\ \psi\in H_o^2(\Omega),\ \text{such that}\ \notag\\
	&\int_\Omega \nabla\phi\cdot\nabla\overline{\psi}+\Delta\phi\cdot\Delta\overline{\psi}dx=\int_\Omega f\cdot\overline{\psi}dx \bigg\},\label{domain1}\\
	\mathbb{A}\phi:=-f,\ \text{where}& \ f \ \text{is given by} \ D(\mathbb{A}), \quad \left\|\phi\right\|_{D(\mathbb{A})}:=\left\|\phi\right\|_{L^2(\Omega)}+\left\|\mathbb{A}\phi\right\|_{L^2(\Omega)}.\label{domain2}
\end{align}
It is easy to see that $\Delta \phi|_{\partial\Omega}=0$ for all $\phi\in D(\mathbb{A})$, and from elliptic regularity theory, it follows that
\begin{equation}\label{domain3}
	D(\mathbb{A})=\left\{\chi\in H^{4}(\Omega): \ \chi|_{\partial\Omega}=\Delta \chi|_{\partial\Omega}=0\right\}=H_o^4(\Omega),\quad\left\|\chi\right\|_{D(\mathbb{A})}\simeq\|\chi\|_{H^4(\Omega)}.
\end{equation}
We also define the linear operator $\mathcal{A}$ with its domain $D(\mathcal{A})$ and the graph norm of $\mathcal{A}$ by
\bse\label{A-1}
\be\label{A-1-1}
\mathcal{A}=\begin{pmatrix}0\ &\mathbb{A}\\ 1\ &0\end{pmatrix}, \quad D(\mathcal{A})= H_o^2(\Omega)\times H_o^4(\Omega),
\ee
\be\label{A-1-2}
\|\mathbf{d}\|_{D(\mathcal{A})}:=\|\mathbf{d}\|_{\mathfrak{X}}+\|\mathcal{A}\mathbf{d}\|_{\mathfrak{X}}\simeq\|d_1\|_{H^2(\Omega)}+\|d_2\|_{H^4(\Omega)}, \quad\mathbf{d}=\left(d_1, d_2\right)\in D(\mathcal{A}),
\ee
\ese 
and have the following generation result for $\mathcal{A}$. The proof can be found in Appendix \ref{AppB}. 
\begin{lemma}\label{generator}
The linear operator $\mathcal{A}$ generates a strongly continuous semigroup $\left\{T(t)\in\mathcal{B}\left(\mathfrak{X}\right): t\in[0, \infty)\right\}$.
\end{lemma}

We are going to study the unique existence of the strict solution for the initial-boundary problems of nonlinear coupled system \eqref{Chap5cp2}.

We set $\tilde{w}(x,t)=w(x,t)-\theta_2$ and $\tilde{w}(t): x\in\Omega\longrightarrow \tilde{w}(x,t)\in\mathbb{R}$, the pinned boundary conditions $\tilde{w}(x,t)=0$,  $\Delta\tilde{w}(x,t)=0$,  $x\in\partial\Omega$, $t\in[0, T]$ and the differential operator $\Delta-\Delta^2$ of the equation \eqref{Chap5cp2-1-2} are incorporated in the operator $\mathbb{A}$ in the generalized definition forms \eqref{domain1} and \eqref{domain2}, and $\mathbb{A}$ defines the pinned realization of the differential operator $\Delta-\Delta^2$.

Using the definition of $\mathbb{A}$, we rewrite the coupled system \eqref{Chap5cp2} as the equation \eqref{Chap54th-LWE-1} with the abstract inhomogeneous term $\tilde{u}=\mathbf{S}_{o}\left(\tilde{v}, \tilde{w}+\theta_2\right)$, $\tilde{v}=\tilde{w}'$:
\be\label{Chap54th-LWE-1}
\tilde{w}^{''}(t)=\mathbb{A}\tilde{w}(t)+[G(\tilde{w})](t)+\beta_p\tilde{u}(t),\quad t\in[0, T],\quad \tilde{w}(0)=\tilde{w}_0,\quad \tilde{w}'(0)=\tilde{v}_0.
\ee
\textcolor{black}{Note that the constant boundary datum $\theta_2$ lifts to a constant function in $\Omega$. Further,} $\tilde{w}'$ and $\tilde{w}^{''}$ respectively denote the first and second derivative of the unknown function $\tilde{w}$ with respect to $t\in [0, T]$, $\tilde{u}(t)=\mathbf{S}_{o}\left(\tilde{v}(t), \tilde{w}(t)+\theta_2\right)$, $\tilde{v}(t)=\tilde{w}'(t)$,
\be\label{Chap5G-2-2}
[G(\tilde{w})](t)=-\frac{\beta_F}{(\tilde{w}(t)+\theta_2)^2}+\beta_p(\theta_1-1),\quad \Phi_0=\left(\tilde{v}_0, \tilde{w}_0\right)\in D(\mathcal{A}).
\ee

\begin{theorem}\label{Chap54th-solu-thm} Let $\kappa=\displaystyle\inf_{x\in{\Omega}}\left\{\tilde{w}_0(x)+\theta_2\right\}$ and $C=C(\Omega)>0$ be  constants, $\tilde{v}_0\in H^2_o(\Omega)$, $\tilde{w}_0\in H^4_o(\Omega)$.
	Then there exists $T_0>0$,  such that for all $T\in(0, T_0)$, the semilinear evolution equation for unknown function $\left(\tilde{v}, \tilde{w}\right)$ with $\tilde{u}=\mathbf{S}_{o}\left(\tilde{v}, \tilde{w}+\theta_2\right)$,
	\be\label{Chap54th-SWE-1}
	\begin{pmatrix}\tilde{v}'(t) \\ \tilde{w}'(t) \end{pmatrix}=\mathcal{A}\begin{pmatrix}\tilde{v}(t) \\ \tilde{w}(t) \end{pmatrix}+\begin{pmatrix}[G(\tilde{w})](t)+\beta_p\tilde{u}(t), \\ 0 \end{pmatrix}, \ t\in[0, T],\ \begin{pmatrix}\tilde{v}(0) \\ \tilde{w}(0) \end{pmatrix}=\begin{pmatrix}\tilde{v}_0 \\ \tilde{w}_0\end{pmatrix},
	\ee
	admits a unique mild solution $\left(\tilde{v}, \tilde{w}\right)\in  C\left([0, T]; L^2(\Omega)\times H_o^2(\Omega)\right)$, i.e. $\left(\tilde{v}, \tilde{w}\right)$ satisfies
	\be\label{Chap5mild-solu-form}
	\begin{pmatrix}\tilde{v}(t)\\ \tilde{w}(t)\end{pmatrix}=T(t)\begin{pmatrix}\tilde{v}_0\\ \tilde{w}_0\end{pmatrix}+
	\int_0^t\left\{T(t-s)\begin{pmatrix}[G(\tilde{w})](s)+\beta_p\tilde{u}(s)\\ 0\end{pmatrix}\right\}ds.
	\ee
\end{theorem}
\begin{proof}
	We let $T\in(0, \infty)$, which will be specified below, and define a complete metric space  $\mathcal{Z}(T)$ for the metric induced by the norm $\sup_{t\in[0, T]}\|(\tilde{v}(t),\tilde{w}(t))\|_{L^2(\Omega)\times H^{2}(\Omega)}$ as follows:
	\begin{align}
		\mathcal{Z}(T):=\bigg\{&(\tilde{v},\tilde{w})\in C\left([0,T]; L^2(\Omega)\times H_o^2(\Omega)\right):\
		\left(\tilde{v}(0), \tilde{w}(0)\right)=\left(\tilde{v}_0, \tilde{w}_0\right),\notag\\
		&\sup_{t\in[0, T]}\|(\tilde{v}(t)-\tilde{v}_0,\tilde{w}(t)-\tilde{w}_0)\|_{L^2(\Omega)\times H^{2}(\Omega)}\leq r\bigg\},\quad \forall\ r\in\left(0, \frac{\kappa}{2C}\right)\label{Chap5ini-NBD}.
	\end{align}
	Recall that $\mathcal{A}$ generates a strongly continuous semigroup ($C_0$-semigroup)
	\[\left\{T(t)\in\mathcal{B}\left(L^2(\Omega)\times H_o^2(\Omega)\right):\ t\in[0, \infty)\right\},\] and  $[G\left(\tilde{w}\right)](t)=-\beta_F[\tilde{w}(t)+\theta_2]^{-2}+\beta_p\left(\theta_1-1\right)$, we introduce a operator $\Phi$ on $\mathcal{Z}(T)$ by
	\[\left[\Phi(\tilde{v},\tilde{w})\right](t):=T(t)\begin{pmatrix}\tilde{v}_0\\ \tilde{w}_0\end{pmatrix}+\int_0^t\left\{T(t-s)\begin{pmatrix}[G(\tilde{w})](s)+\beta_p\tilde{u}(s)\\ 0\end{pmatrix}\right\}ds,\quad \forall\ t\in [0,T].\]
	We notice that
	\[\left[\Phi(\tilde{v},\tilde{w})\right](0)=T(0)\begin{pmatrix}\tilde{v}_0\\ \tilde{w}_0\end{pmatrix}=\begin{pmatrix}\tilde{v}_0\\ \tilde{w}_0\end{pmatrix}\in D(\mathcal{A}).\]
	According to  Lemma 1.3 of Chapter II in \cite{EK},
	\[T(t)\begin{pmatrix}\tilde{v}_0\\ \tilde{w}_0\end{pmatrix}\in D(\mathcal{A}),\ \forall\ t\in[0, T].\]
	Since $(\tilde{v},\tilde{w})\in \mathcal{Z}(T)$, $\tilde{u}\in C\left([0, T]; H_0^1(\Omega)\right)$ such that $G(\tilde{w})+\beta_p\tilde{u}\in C([0, T]; L^2(\Omega))$, hence
	\[\begin{pmatrix}[G(\tilde{w})](t)+\beta_p\tilde{u}(t)\\ 0\end{pmatrix},\quad \int_0^t\left\{T(t-s)\begin{pmatrix}[G(\tilde{w})](s)+\beta_p\tilde{u}(s)\\ 0\end{pmatrix}\right\}ds\in L^2(\Omega)\times H_o^2(\Omega).\]
	Therefore, $\Phi$ is a nonlinear operator which maps $\mathcal{Z}(T)$ into $C\left([0,T]; L^2(\Omega)\times H_o^2(\Omega)\right)$:
	\[\Phi: \ \mathcal{Z}(T)\rightarrow\ C\left([0,T]; L^2(\Omega)\times H_o^2(\Omega)\right). \]
	We next show that there exists a unique mild solution $(\tilde{v},\tilde{w})\in \mathcal{Z}(T)$ of the semilinear evolution equation \eqref{Chap54th-SWE-1} which is a fixed point of $\Phi$ on $\mathcal{Z}(T)$.
	
	We denote by $M_0=\sup_{t\in [0, \infty )}\|T(t)\|_{\mathcal{B}\left(L^2(\Omega)\times H_o^2(\Omega)\right)}$ an operator norm of $\left\{T(t)\right\}_{0\leq t<\infty}$ on the space $L^2(\Omega)\times H_o^2(\Omega)$. If $(\tilde{v}_1,\tilde{w}_1),\ (\tilde{v}_2,\tilde{w}_2)\in \mathcal{Z}(T)$, then $\tilde{v}_1(t)\in L^2(\Omega)$, $\tilde{w}_1(t)\in H_o^2(\Omega)$,  thus $[G(\tilde{w}_1)](t)\in H^2(\Omega)$, $[G(\tilde{w}_1)](t)-[G(\tilde{w}_2)](t)\in H^2(\Omega)$, $\forall\ t\in[0, T]$.
	
	We set $\tilde{u}_1(t)=\mathbf{S}_o\left(\tilde{v}_1(t), \tilde{w}_1(t)+\theta_2\right)$ and $\tilde{u}_2(t)=\mathbf{S}_o\left(\tilde{v}_2(t), \tilde{w}_2(t)+\theta_2\right)$. The estimates \eqref{Lip-G}, \eqref{Lip-G-1} of Corollary \ref{estimates} and estimates \eqref{Chap5time-reg-u-H2}, \eqref{Chap5time-Lip-con-sol3} imply
	\begin{align}
		\sup_{t\in[0, T]}\left\|[G(\tilde{w}_1)](t)-[G(\tilde{w}_2)](t)\right\|_{L^{2}(\Omega)}\leq L_G\sup_{t\in[0, T]}\left\| \tilde{w}_1(t)-\tilde{w}_2(t)\right\|_{H^{2}(\Omega)},\label{Chap5G1}
	\end{align}
	\begin{align}
		\sup_{t\in[0, T]}\left\|\tilde{u}_1(t)-\tilde{u}_2(t)\right\|_{L^{2}(\Omega)}
		&\leq C_o\sup_{t\in[0, T]}\left\|\tilde{v}_1(t)-\tilde{v}_2(t)\right\|_{L^2(\Omega)}\notag\\
		&+C_o^*\sup_{t\in[0, T]}\left\|\tilde{v}_2(t)\right\|_{L^2(\Omega)}\sup_{t\in[0, T]}\left\|\tilde{w}_1(t)-\tilde{w}_2(t)\right\|_{H^2(\Omega)}.\label{Chap5U1}
	\end{align}
	\begin{align}
		\sup_{t\in[0, T]}\left\|[G(\tilde{w}_1)](t)-G(\tilde{w}_0)\right\|_{L^2(\Omega)}\leq\sup_{t\in[0, T]}\left\|[G(\tilde{w}_1)](t)-G(\tilde{w}_0)\right\|_{H^2(\Omega)}\leq L_Gr\label{Chap5G2},
	\end{align}
	\begin{align}
		\sup_{t\in[0, T]}\left\|\tilde{u}(t)\right\|_{L^2(\Omega)}\leq&\sup_{t\in[0, T]}\left\|\tilde{u}(t)\right\|_{H^1(\Omega)}\leq C_o\left[\left\|\tilde{v}_0\right\|_{L^2(\Omega)}+\frac{\kappa}{2C}\right]\label{Chap5U2}.
	\end{align}
	Setting
	\be\label{Chap5L-G}
	L_G^*=L_G+\beta_pC_o+\beta_pC_o^*\left(\left\|\tilde{v}_0\right\|_{L^2(\Omega)}+\frac{\kappa}{2C}\right),
	\ee
	estimates \eqref{Chap5G1}, \eqref{Chap5U1}, \eqref{Chap5G2}, \eqref{Chap5U2} and the bounded property of the strongly continuous semigroup $\{T(t)\in\mathcal{B}\left(L^2(\Omega)\times H_o^2(\Omega)\right):\ t\geq 0\}$ imply
	\begin{align}
		&\sup_{t\in[0, T]}\left\|[\Phi(\tilde{v}_1,\tilde{w}_1)](t)-[\Phi(\tilde{v}_2,\tilde{w}_2)](t)\right\|_{L^2(\Omega)\times H^{2}(\Omega)}\notag\\
		=&\sup_{t\in[0, T]}\left\|\int_0^tT(t-s)\begin{pmatrix}[G(\tilde{w}_1)](s)-[G(\tilde{w}_2)](s)+\beta_p\left[\tilde{u}_1(s)-\tilde{u}_2(s)\right]\\ 0\end{pmatrix}ds\right\|_{L^{2}(\Omega)\times H^{2}(\Omega)}\notag\\
		\leq&TM_0\sup_{t\in[0, T]}\left\|[G(\tilde{w}_1)](t)-[G(\tilde{w}_2)](t)\right\|_{L^{2}(\Omega)}+TM_0\beta_p\sup_{t\in[0, T]}\left\|\tilde{u}_1(t)-\tilde{u}_2(t)\right\|_{L^{2}(\Omega)}\notag\\
		\leq&TM_0L_G^*\sup_{t\in[0, T]}\left\|\begin{pmatrix}\tilde{v}_1(t)-\tilde{v}_2(t)\\ \tilde{w}_1(t)-\tilde{w}_2(t)\end{pmatrix}\right\|_{L^{2}(\Omega)\times H^{2}(\Omega)}\label{Chap5difference1},
	\end{align}
	\begin{align}
		&\sup_{t\in[0, T]}\left\|\left[\Phi(\tilde{v}_1,\tilde{w}_1)\right](t)-\begin{pmatrix}\tilde{v}_0\\ \tilde{w}_0\end{pmatrix}\right\|_{L^2(\Omega)\times H^{2}(\Omega)}\notag\\
		=&\sup_{t\in[0, T]}\left\|T(t)\begin{pmatrix}\tilde{v}_0\\ \tilde{w}_0\end{pmatrix}-\begin{pmatrix}\tilde{v}_0\\ \tilde{w}_0\end{pmatrix}+\int_0^t\left\{T(t-s)\begin{pmatrix}[G(\tilde{w}_1)](s)+\beta_p\tilde{u}(s)\\ 0\end{pmatrix}\right\}ds\right\|_{L^2(\Omega)\times H^{2}(\Omega)}\notag\\
		\leq&\sup_{t\in[0, T]}\left\|T(t)\begin{pmatrix}\tilde{v}_0\\ \tilde{w}_0\end{pmatrix}-\begin{pmatrix}\tilde{v}_0\\ \tilde{w}_0\end{pmatrix}\right\|_{L^2(\Omega)\times H^{2}(\Omega)}+TM_0\sup_{t\in[0, T]}\left\|\tilde{u}(t)\right\|_{L^2(\Omega)}\notag\\
		+& TM_0\sup_{t\in[0, T]}\left\|[G(\tilde{w}_1)](t)-G(\tilde{w}_0)\right\|_{L^2(\Omega)}+ TM_0\left\|G(\tilde{w}_0)\right\|_{ L^{2}(\Omega)}\notag\\
		\leq&\sup_{t\in[0, T]}\left\|T(t)\begin{pmatrix}\tilde{v}_0\\ \tilde{w}_0\end{pmatrix}-\begin{pmatrix}\tilde{v}_0\\ \tilde{w}_0\end{pmatrix}\right\|_{L^2(\Omega)\times H^{2}(\Omega)}+TM_0\left[C_o\left(\left\|\tilde{v}_0\right\|_{L^2(\Omega)}+\frac{\kappa}{2C}\right)+L_G r\right]\notag\\
		+&TM_0\left\|G(\tilde{w}_0)\right\|_{ H^{2}(\Omega)}\label{Chap5estimate of phi-1}.
	\end{align}
	Because $\{T(t)\in\mathcal{B}\left(L^2(\Omega)\times H_o^2(\Omega)\right):\ t\geq 0\}$ is a strongly continuous semigroup,  according to the definition of strong continuity, for $(\tilde{v}_0, \tilde{w}_0)\in D(\mathcal{A})$ and given constant $r\in\left(0, \frac{\kappa}{2C}\right)$, there exists $\delta_o=\delta_o(r)>0$, such that if $0<t\leq\delta_o$, then
	\be\label{Chap5semigp-continuity}
	0<\left\|T(t)\begin{pmatrix}\tilde{v}_0\\ \tilde{w}_0\end{pmatrix}-\begin{pmatrix}\tilde{v}_0\\ \tilde{w}_0\end{pmatrix}\right\|_{L^2(\Omega)\times H^{2}(\Omega)}\leq\frac{r}{2},
	\ee
	Since $r\in\left(0, \frac{\kappa}{2C}\right)$ and $C=C(\Omega)$ is a constant, $\delta_o$ depends on $\kappa$ and $\Omega$, i.e. $\delta_o=\delta_o(\kappa, \Omega)$.
	
	For fixed small $r\in\left(0, \frac{\kappa}{2C}\right)$, then there exists a number $T_0>0$,
	\be\label{Chap5T-0}
	T_0=\min\left\{\delta_o,\ \frac{1}{2M_0L^*_G},\
	\frac{\kappa}{2M_0}\left[\left(L_G+C_o\right)\kappa+2C\left(\left\|G(\tilde{w}_0)\right\|_{ H^{2}(\Omega)}+C_o\left\|\tilde{v}_0\right\|_{L^2(\Omega)}\right)\right]^{-1}\right\},
	\ee
	such that for every $T\in (0,T_0)$, it follows that
	\[\sup_{t\in[0, T]}\left\|\left[\Phi(\tilde{v}_1,\tilde{w}_1)\right](t)-\begin{pmatrix}\tilde{v}_0\\ \tilde{w}_0\end{pmatrix}\right\|_{L^2(\Omega)\times H^{2}(\Omega)}\leq r,\]
	\[\sup_{t\in[0, T]}\|[\Phi(\tilde{v}_1,\tilde{w}_1)](t)-[\Phi(\tilde{v}_2,\tilde{w}_2)](t)\|_{L^2(\Omega)\times H^{2}(\Omega)}
	\leq\frac{1}{2}\sup_{t\in[0, T]}\left\|\begin{pmatrix}\tilde{v}_1(t)-\tilde{v}_2(t)\\ \tilde{w}_1(t)-\tilde{w}_2(t)\end{pmatrix}\right\|_{L^{2}(\Omega)\times H^{2}(\Omega)}.\]
	Hereby $\Phi(\tilde{v}_1,\tilde{w}_1)\in \mathcal{Z}(T)$ for $(\tilde{v}_1,\tilde{w}_1)\in\mathcal{Z}(T)$,  $\Phi(\tilde{v}, \tilde{w})$ is Lipschitz continuous on the bounded set $\mathcal{Z}(T)$ with Lipschitz constant smaller than or equal to $\frac{1}{2}$, and $\Phi(\tilde{v}, \tilde{w})$ is a contractive mapping of $\mathcal{Z}(T)$ into itself.
	
	According to the Banach fixed point theorem, for each $T\in(0, T_0)$, there exists a unique fixed point $(\tilde{v}_T, \tilde{w}_T)\in \mathcal{Z}(T)$, such that $(\tilde{v}_T,\tilde{w}_T)=\Phi(\tilde{v}_T,\tilde{w}_T)$.
	
	Hence, $(\tilde{v}_T, \tilde{w}_T)\in \mathcal{Z}(T)$ is the unique mild solution of the semilinear evolution equation \eqref{Chap54th-SWE-1} on $[0, T]$, and $(\tilde{v}_T,\tilde{w}_T)$ satisfies the integral formulation \eqref{Chap5mild-solu-form}. Due to the uniqueness of the fixed point, we set $(\tilde{v}, \tilde{w})=(\tilde{v}_{T},\tilde{w}_{T})$ and note that $(\tilde{v}_T, \tilde{w}_T)$ is the restriction $(\tilde{v}|_{[0, T]}, \tilde{w}|_{[0, T]})\in\mathcal{Z}(T)$ of $(\tilde{v}, \tilde{w})$. As a result, the assertion is proved.
\end{proof}
\begin{corollary}\label{Chap5Lip-mild-solu}
	Under the assumptions of Theorem \ref{Chap54th-solu-thm}, the mild solution of the semilinear evolution equation \eqref{Chap54th-SWE-1},
	$\left(\tilde{v}, \tilde{w}\right): [0, T]\rightarrow L^2(\Omega)\times H_o^2(\Omega)$,
	defined by the integral form \eqref{Chap5mild-solu-form}, is locally Lipschitz continuous with respect to $t\in[0, T]$, i.e. $\forall\ h\in(0, T]$,
	\be\label{Chap5Lip-mild-solu-inquality}
	\sup_{0\leq t<t+h\leq T}\left\|\begin{pmatrix}\tilde{v}(t+h)-\tilde{v}(t)\\ \tilde{w}(t+h)-\tilde{w}(t)\end{pmatrix}\right\|_{L^2(\Omega)\times H^2(\Omega)}\leq L_Vh.
	\ee
	Here $L_V$ is a Lipschitz constant depending on $\beta_F$, $\beta_p$, $T_0$, $\kappa$, $\Omega$,
	$\|\tilde{w}_0\|_{H^2(\Omega)}$, $\left\|\left(\tilde{v}_0, \tilde{w}_0\right)\right\|_{D(\mathcal{A})}$,  $M_0=\sup_{t\in[0, \infty)}\left\|T(t)\right\|_{\mathcal{B}(L^2(\Omega)\times H_o^2(\Omega))}$.
\end{corollary}
\begin{proof}
	Take $0\leq t<t+h\leq T$. Equation \eqref{Chap5mild-solu-form} leads to
	\begin{align}
		\begin{pmatrix}\tilde{v}(t+h)-\tilde{v}(t)\\ \tilde{w}(t+h)-\tilde{w}(t)\end{pmatrix}=&T(t)\left[T(h)\begin{pmatrix}\tilde{v}_0\\ \tilde{w}_0\end{pmatrix}-\begin{pmatrix}\tilde{v}_0\\ \tilde{w}_0\end{pmatrix}\right]+\int_0^hT(t+h-s)\begin{pmatrix}[G(\tilde{w})](s)+\beta_p\tilde{u}(s)\\ 0\end{pmatrix}ds\notag\\
		+&\int_0^tT(t-s)\begin{pmatrix}[G(\tilde{w})](s+h)-[G(\tilde{w})](s)+\beta_p[\tilde{u}(s+h)-\tilde{u}(s)]\\ 0\end{pmatrix}ds\notag\\
		=& \int_0^hT(t+s)\mathcal{A}\begin{pmatrix}\tilde{v}_0\\ \tilde{w}_0\end{pmatrix}ds+\int_0^hT(t+h-s)\begin{pmatrix}[G(\tilde{w})](s)+\beta_p\tilde{u}(s)\\ 0\end{pmatrix}ds\notag\\
		+&\int_0^tT(t-s)\begin{pmatrix}[G(\tilde{w})](s+h)-[G(\tilde{w})](s)+\beta_p[\tilde{u}(s+h)-\tilde{u}(s)]\\ 0\end{pmatrix}ds.\label{Chap5diff-mild-solution}
	\end{align}
	Notice that
	\be\label{Chap5Lip-est-1}
	M_0=\sup_{t\in[0, \infty)}\left\|T(t)\right\|_{\mathcal{B}\left(L^2(\Omega)\times H_o^2(\Omega)\right)}, \quad \left\|\mathcal{A}\begin{pmatrix}\tilde{v}_0\\ \tilde{w}_0\end{pmatrix}\right\|_{L^2(\Omega)\times H^2(\Omega)}\leq \left\|\left(\tilde{v}_0, \tilde{w}_0\right)\right\|_{D(\mathcal{A})}.
	\ee
	Because the semilinear evolution equation \eqref{Chap54th-SWE-1} has a unique mild solution $\left(\tilde{v}, \tilde{w}\right)\in \mathcal{Z}(T)$, by using the estimate \eqref{Lip-G-1} in Corollary \ref{Lip-G-Lem}, we have
	\begin{align}
		\sup_{t\in[0, T]}\left\|\begin{pmatrix}[G(\tilde{w})](t)+\beta_p\tilde{u}(t)\\ 0\end{pmatrix}\right\|_{L^2(\Omega)\times H^2(\Omega)}
		&\leq \sup_{t\in[0, T]}\left\|[G(\tilde{w})](t)+\beta_p\tilde{u}(t)\right\|_{H^2(\Omega)}\notag\\
		&\leq \sup_{t\in[0, T]}\left\|[G(\tilde{w})](t)-G(\tilde{w}_0)\right\|_{H^2(\Omega)}\notag\\
		&+\left\|G(\tilde{w}_0)\right\|_{H^{2}(\Omega)}+\beta_p\sup_{t\in[0, T]}\left\|\tilde{u}(t)\right\|_{H^2(\Omega)}\notag\\
		&\leq L_G r+\left\|G(\tilde{w}_0)\right\|_{H^{2}(\Omega)}+\beta_pC_o\sup_{t\in[0, T]}\left\|\tilde{v}(t)\right\|_{L^{2}(\Omega)}\notag\\
		&\leq C_\theta\label{Chap5Lip-est-2}.
	\end{align}
	Here $C_\theta=\frac{\kappa L_G}{2C}+\left\|G(\tilde{w}_0)\right\|_{H^{2}(\Omega)}+\beta_pC_o\left[\left\|\tilde{v}_0\right\|_{L^2(\Omega)}+\frac{\kappa}{2C}\right]$ and $L_G$ is given by Corollary \ref{Lip-G-Lem}.
	
	As $\tilde{u}=\mathbf{S}_o\left(\tilde{v}, \tilde{w}+\theta_2\right)\in C\left([0, T]; H_0^1(\Omega)\right)$ and \eqref{Chap5time-Lip-con-sol4}, then, $\forall\ 0\leq s< s+h \leq t\leq T$,
	\begin{align}
		\sup_{t\in[0, T]}\int_0^t\left\|\tilde{u}(s+h)-\tilde{u}(s)\right\|_{L^2(\Omega)}ds\leq& \sup_{t\in[0, T]}\int_0^t\left\|\tilde{u}(s+h)-\tilde{u}(s)\right\|_{H^1(\Omega)}ds\notag\\
		\leq&C_\alpha\sup_{t\in[0, T]}\int_0^t\left\|\begin{pmatrix}\tilde{v}(s+h)-\tilde{v}(s)\\  \tilde{w}(s+h)-\tilde{w}(s)\end{pmatrix}\right\|_{L^2(\Omega)\times H^2(\Omega)}ds\label{Chap5Holderinu}.
	\end{align}
	Here, $C_\alpha=C_o^*\left(\left\|\tilde{v}_0\right\|_{L^2(\Omega)}+\frac{\kappa}{2C}\right)+C_o$.
	
	As $\tilde{w}(s+h)$, $\tilde{w}(s)\in H_o^2(\Omega)$, $\forall\ 0\leq s< s+h \leq t\leq T$, $[G(\tilde{w})](s+h)-[G(\tilde{w})](s)\in H^2(\Omega)$, according to inequality \eqref{Holdercontinuous} of Corollary \ref{Lip-G-Lem}, setting $C_\beta=\beta_pC_\alpha+M_0L_G$, we obtain
	\begin{align}
		&\left\|\int_0^tT(t-s)\begin{pmatrix}[G(\tilde{w})](s+h)-[G(\tilde{w})](s)+\beta_p[\tilde{u}(s+h)-\tilde{u}(s)]\\ 0\end{pmatrix}ds\right\|_{L^2(\Omega)\times H^2(\Omega)}\notag\\
		=& M_0\int_0^t\left\|[G(\tilde{w})](s+h)-[G(\tilde{w})](s)+\beta_p[\tilde{u}(s+h)-\tilde{u}(s)]\right\|_{L^2(\Omega)}ds\notag\\
		\leq& M_0\int_0^t\beta_p\left\|\tilde{u}(s+h)-\tilde{u}(s)\right\|_{L^2(\Omega)}
		+\left\|[G(\tilde{w})](s+h)-[G(\tilde{w})](s)\right\|_{H^2(\Omega)}ds\notag\\
		\leq& C_\beta\int_0^t\left\|\begin{pmatrix}\tilde{v}(s+h)-\tilde{v}(s)\\  \tilde{w}(s+h)-\tilde{w}(s)\end{pmatrix}\right\|_{L^2(\Omega)\times H^2(\Omega)}ds\label{Chap5diff-mild-solution-1}.
	\end{align}
	Combining \eqref{Chap5diff-mild-solution}, \eqref{Chap5Lip-est-1}, \eqref{Chap5Lip-est-2} and \eqref{Chap5diff-mild-solution-1} gives
	\begin{align}
		\left\|\begin{pmatrix}\tilde{v}(t+h)-\tilde{v}(t)\\ \tilde{w}(t+h)-\tilde{w}(t)\end{pmatrix}\right\|_{L^2(\Omega)\times H^2(\Omega)}&\leq hM_0\left\|\left(\tilde{v}_0, \tilde{w}_0\right)\right\|_{D(\mathcal{A})}+hM_0C_\theta\notag\\
		&+C_\beta\int_0^t\left\|\begin{pmatrix}\tilde{v}(s+h)-\tilde{v}(s)\\  \tilde{w}(s+h)-\tilde{w}(s)\end{pmatrix}\right\|_{L^2(\Omega)\times H^2(\Omega)}ds\label{Chap5diff-mild-solution-2}.
	\end{align}
	Gronwall's inequality then implies that
	\[\left\|\begin{pmatrix}\tilde{v}(t+h)-\tilde{v}(t)\\ \tilde{w}(t+h)-\tilde{w}(t)\end{pmatrix}\right\|_{L^2(\Omega)\times H^2(\Omega)}\leq M_0\left(\left\|\left(\tilde{v}_0, \tilde{w}_0\right)\right\|_{D(\mathcal{A})}+C_\theta\right)\left(e^{M_0L_GT_0}\right)h.\]
	Thus, \eqref{Chap5Lip-mild-solu-inquality} holds for all $h\in(0, T]$ by setting $L_V=M_0\left[\left\|\left(\tilde{v}_0, \tilde{w}_0\right)\right\|_{D(\mathcal{A})}+C_\theta\right]\left[e^{M_0L_GT_0}\right].$
\end{proof}

\begin{theorem}\label{Chap54th-mild-solution-cor}
	Let $\tilde{v}_0\in H^2_o(\Omega)$ and $\tilde{w}_0\in H^4_o(\Omega)$. Then the mild solution $\left(\tilde{v}, \tilde{w}\right)$ of the semilinear evolution equation \eqref{Chap54th-SWE-1}, defined by the integral form \eqref{Chap5mild-solu-form}, is the strict solution of   equation \eqref{Chap54th-SWE-1} and
	\[\left(\tilde{v}, \tilde{w}\right)\in C^1\left([0, T_0); L^2(\Omega)\times H_o^2(\Omega)\right)\cap C\left([0, T_0); H_o^2(\Omega)\times H_o^4(\Omega)\right).\]
\end{theorem}
\begin{proof}
	Let $T\in(0, T_0)$ and $\left(\tilde{v}, \tilde{w}\right)$ be the mild solution of the semilinear evolution equation \eqref{Chap54th-SWE-1} defined by \eqref{Chap5mild-solu-form}. Set $\tilde{u}(t)=\mathbf{S}_o\left(\tilde{v}(t), \tilde{w}(t)+\theta_2\right)$. We first prove the linear non-autonomous problem
	\be\label{Chap5LNP}
	\begin{pmatrix}\tilde{p}(t)\\ \tilde{q}(t)\end{pmatrix}=T(t)\left(\begin{pmatrix}G_0\\ 0\end{pmatrix}+\mathcal{A}\begin{pmatrix}\tilde{v}_0\\ \tilde{w}_0\end{pmatrix}\right)+\int_0^tT(t-s)\begin{pmatrix}[\mathcal{H}_1(\tilde{q})](s)+\mathcal{H}_2(\tilde{p},\tilde{q})](s)\\ 0\end{pmatrix}ds,
	\ee
	can be solved for all $t\in[0, T]$. Here $G_0=G(\tilde{w}_0)+\beta_p\tilde{u}_0$, $G(\tilde{w}_0)=[G(\tilde{w})](0)$,
	\[\tilde{u}_0=\tilde{u}(0)=\mathbf{S}_o\left(\tilde{v}(0), \tilde{w}(0)+\theta_2\right)=\mathbf{S}_o(v_0, w_0),\]
	\be\label{Chap5H-nonlinearity1}
	[\mathcal{H}_1(\tilde{q})](s)=\frac{2\beta_F\tilde{q}(s)}{\left[\tilde{w}(s)+\theta_2\right]^{3}}=[G'(\tilde{w}(s))]q(s)=[G'(\tilde{w})q](s),\quad s\in [0, t],
	\ee
	\be\label{Chap5H-nonlinearity2}
	[\mathcal{H}_2(\tilde{p}, \tilde{q})](s)=[D_v\mathbf{S}_o\left(v(s), w(s)\right)]\tilde{p}(s)+[D_w\mathbf{S}_o\left(v(s), w(s)\right)]\tilde{q}(s),\quad s\in [0, t].
	\ee
	We define a nonlinear operator $\Psi$ by
	\[\Psi:\ C\left([0, T]; L^2(\Omega)\times H_o^2(\Omega)\right)\longrightarrow C\left([0, T]; L^2(\Omega)\times H_o^2(\Omega)\right),\]
	\[\left[\Psi\left(\tilde{p}, \tilde{q}\right)\right](t)=T(t)\left(\begin{pmatrix}G_0\\ 0\end{pmatrix}+\mathcal{A}\begin{pmatrix}\tilde{v}_0\\ \tilde{w}_0\end{pmatrix}\right)+\int_0^tT(t-s)\begin{pmatrix}[\mathcal{H}_1(\tilde{q})](s)+[\mathcal{H}_2(\tilde{p}, \tilde{q})](s)\\ 0\end{pmatrix}ds.\]
	For any $\left(\tilde{p}_1, \tilde{q}_1\right),\ \left(\tilde{p}_2, \tilde{q}_2\right)\in C\left([0, T]; L^2(\Omega)\times H_o^2(\Omega)\right),\ \tilde{w}\in C\left([0, T]; H_o^2(\Omega)\right)$, then
	\[\mathcal{H}_1(\tilde{q}_1)-\mathcal{H}_1(\tilde{q}_2)=\frac{2\beta_F}{\left[\tilde{w}+\theta_2\right]^{3}}\left[\tilde{q}_1-\tilde{q}_2\right]=G'(\tilde{w})(q_1-q_2)\in C\left([0, T]; H_o^2(\Omega)\right).\]
	\begin{align}
		\mathcal{H}_2(\tilde{p}_1, \tilde{q}_1)-\mathcal{H}_2(\tilde{p}_2, \tilde{q}_2)=&\beta_p\left([D_v\mathbf{S}_o\left(\tilde{v}, \tilde{w}+\theta_2\right)]\tilde{p}_1-[D_v\mathbf{S}_o\left(\tilde{v}, \tilde{w}+\theta_2\right)]\tilde{p}_2\right)\notag\\
		+&\beta_p\left([D_w\mathbf{S}_o\left(\tilde{v}, \tilde{w}+\theta_2\right)]\tilde{q}_1-[D_w\mathbf{S}_o\left(\tilde{v}, \tilde{w}+\theta_2\right)]\tilde{q}_2\right)\notag\\
		\in&C\left([0, T]; H_0^1(\Omega)\right)\notag.
	\end{align}
	Hence,  the estimate \eqref{Lip-G-2} of Fr\'{e}chet derivative $G'\left(\tilde{w}\right)q$ from Corollary \ref{Lip-G-Lem} implies
	\begin{align}
		\sup_{t\in[0, T]}\left\|[\mathcal{H}_1(\tilde{q}_1)](t)-[\mathcal{H}_1(\tilde{q}_2)](t)\right\|_{L^2(\Omega)}=&\sup_{t\in[0, T]}\left\|[G'(\tilde{w})(\tilde{q}_1-\tilde{q}_2)](t)\right\|_{H^2(\Omega)}\notag\\
		\leq&L_G\sup_{t\in[0, T]}\left\|\begin{pmatrix}\tilde{p}_1(t)-\tilde{p}_2(t)\\ \tilde{q}_1(t)-\tilde{q}_2(t)\end{pmatrix}\right\|_{L^2(\Omega)\times H^2(\Omega)}\label{Chap5G3}.
	\end{align}
	Estimates \eqref{Chap5Fre-Lip-continuity-v}, \eqref{Chap5Fre-Lip-continuity-w} from Lemma \ref{Chap5Fre-Lip-continuity} give
	\begin{align}
		&\sup_{t\in[0, T]}\left\|[\mathcal{H}_2(\tilde{p}_1, \tilde{q}_1)](t)-[\mathcal{H}_2(\tilde{p}_2, \tilde{q}_2)](t)\right\|_{L^2(\Omega)}\notag\\
		\leq&\beta_p\left(C_o\sup_{t\in[0, T]}\left\|\tilde{p}_1(t)-\tilde{p}_2(t)\right\|_{L^2(\Omega)}+C_o^*\sup_{t\in[0, T]}\left\|\tilde{v}(t)\right\|_{L^2(\Omega)}\sup_{t\in[0, T]}\left\|\tilde{q}_1(t)-\tilde{q}_2(t)\right\|_{H^2(\Omega)}\right)\notag\\
		\leq&\beta_p\left[C_o+C_o^*\left(\left\|\tilde{v}_0\right\|_{L^2(\Omega)}+\frac{\kappa}{2C}\right)\right]\sup_{t\in[0, T]}\left\|\begin{pmatrix}\tilde{p}_1(t)-\tilde{p}_2(t)\\ \tilde{q}_1(t)-\tilde{q}_2(t)\end{pmatrix}\right\|_{L^2(\Omega)\times H^2(\Omega)}\label{Chap5U3}.
	\end{align}
	Recall $M_0=\sup_{t\in[0, \infty)}\left\|T(t)\right\|_{\mathcal{B}\left(L^2(\Omega)\times H_o^2(\Omega)\right)}$. The definition \eqref{Chap5T-0} of $T_0$, the definition \eqref{Chap5L-G} of $L_G^*$, \eqref{Chap5G3} and \eqref{Chap5U3} imply $\Psi$ is a contractive mapping on $C\left([0, T]; L^2(\Omega)\times H_o^2(\Omega)\right)$ by
	\begin{align}
		&\left\|\left[\Psi\left(\tilde{p}_1, \tilde{q}_1\right)\right](t)-\left[\Psi\left(\tilde{p}_2, \tilde{q}_2\right)\right](t)\right\|_{L^2(\Omega)\times H^2(\Omega)}\notag\\
		=&\left\|\int_0^tT(t-s)\begin{pmatrix}[\mathcal{H}_1(\tilde{q}_1)](s)-[\mathcal{H}_1(\tilde{q}_2)](s)+\mathcal{H}_2(\tilde{p}_1, \tilde{q}_1)](s)-\mathcal{H}_2(\tilde{p}_2, \tilde{q}_2)](s)\\ 0\end{pmatrix}ds\right\|_{L^2(\Omega)\times H^2(\Omega)}\notag\\
		\leq& TM_0\sup_{t\in[0, T]}\left\{\left\|[\mathcal{H}_1(\tilde{q}_1)](t)-[\mathcal{H}_1(\tilde{q}_2)](t)\right\|_{L^2(\Omega)}+\left\|[\mathcal{H}_2(\tilde{p}_1, \tilde{q}_1)](t)-[\mathcal{H}_2(\tilde{p}_2, \tilde{q}_2)](t)\right\|_{L^2(\Omega)}\right\}\notag\\
		\leq&TM_0L_G^*\sup_{t\in[0, T]}\left\|\begin{pmatrix}\tilde{p}_1(t)-\tilde{p}_2(t)\\ \tilde{q}_1(t)-\tilde{q}_2(t)\end{pmatrix}\right\|_{L^2(\Omega)\times H^2(\Omega)}\leq\frac{1}{2}\sup_{t\in[0, T]}\left\|\begin{pmatrix}\tilde{p}_1(t)-\tilde{p}_2(t)\\ \tilde{q}_1(t)-\tilde{q}_2(t)\end{pmatrix}\right\|_{L^2(\Omega)\times H^2(\Omega)}\notag.
	\end{align}
	According to the Banach fixed point theorem, there exists a unique fixed point $(\tilde{p}, \tilde{q})$ in $C\left([0, T]; L^2(\Omega)\times H_o^2(\Omega)\right)$, such that $(\tilde{p}, \tilde{q})=\Psi(\tilde{p}, \tilde{q})$. Hereby, the $\mathbb{R}$-linear non-autonomous problem \eqref{Chap5LNP} can be solved for $t\in[0, T]$.
	
	We next prove that $(\tilde{p}, \tilde{q})$ is the time derivative of the mild solution $(\tilde{v}, \tilde{w})$. Recall $\tilde{u}(t)=\mathbf{S}_o\left(\tilde{v}(t), \tilde{w}(t)+\theta_2\right)$, $\forall\ t\in[0, T]$, $G_0=G(\tilde{w}_0)+\beta_p\tilde{u}_0=[G(\tilde{w})](0)+\beta_p\tilde{u}(0)$, and let $0\leq t<t+h\leq T$ for some $h\in(0, T]$, equations \eqref{Chap5mild-solu-form} and \eqref{Chap5LNP} imply that
	\begin{align}
		E(h,t)&:=\frac{1}{h}\begin{pmatrix}\tilde{v}(t+h)-\tilde{v}(t)\\ \tilde{w}(t+h)-\tilde{w}(t)\end{pmatrix}-\begin{pmatrix}\tilde{p}(t)\\ \tilde{q}(t)\end{pmatrix}\notag\\
		&=T(t)\frac{1}{h}\left(T(h)-I\right)\begin{pmatrix}\tilde{v}_0\\ \tilde{w}_0\end{pmatrix}-T(t)\mathcal{A}\begin{pmatrix}\tilde{v}_0\\ \tilde{w}_0\end{pmatrix}\notag\\
		&+\frac{1}{h}\int_0^h\left\{T(t+h-s)\begin{pmatrix}[G(\tilde{w})](s)+\beta_p\tilde{u}(s)\\ 0\end{pmatrix}\right\}ds-T(t)\begin{pmatrix}G_0\\ 0\end{pmatrix}\notag\\
		&+\int_0^tT(t-s)\begin{pmatrix}\frac{1}{h}\left\{[G(\tilde{w})](s+h)-[G(\tilde{w})](s)\right\}-[\mathcal{H}_1(\tilde{q})](s)\\ 0\end{pmatrix}ds\notag\\
		&+\int_0^tT(t-s)\begin{pmatrix}\beta_p\left[\frac{1}{h}\left\{\tilde{u}(s+h)-\tilde{u}(s)\right\}-[\mathcal{H}_2(\tilde{p}, \tilde{q})](s)\right]\\ 0\end{pmatrix}ds\notag.
	\end{align}
	Let
	\[E^{(1)}(h,t):=T(t)\frac{1}{h}\left(T(h)-I\right)\begin{pmatrix}\tilde{v}_0\\ \tilde{w}_0\end{pmatrix}-T(t)\mathcal{A}\begin{pmatrix}\tilde{v}_0\\ \tilde{w}_0\end{pmatrix},\]
	\[E^{(2)}(h,t):=\frac{1}{h}\int_0^h\left\{T(t+h-s)\begin{pmatrix}[G(\tilde{w})](s)+\beta_p\tilde{u}(s)\\ 0\end{pmatrix}\right\}ds-T(t)\begin{pmatrix}G_0\\ 0\end{pmatrix},\]
	\begin{align}
		E^{(3)}(h,t)&:=\int_0^tT(t-s)\begin{pmatrix}\frac{1}{h}\left\{[G(\tilde{w})](s+h)-[G(\tilde{w})](s)\right\}-[\mathcal{H}_1(\tilde{q})](s)\\ 0\end{pmatrix}ds\notag\\
		&+\int_0^tT(t-s)\begin{pmatrix}\beta_p\left[\frac{1}{h}\left\{\tilde{u}(s+h)-\tilde{u}(s)\right\}-[\mathcal{H}_2(\tilde{p}, \tilde{q})](s)\right]\\ 0\end{pmatrix}ds\notag.
	\end{align}
	We initially notice that
	\begin{align}
		\lim_{h\rightarrow 0}\left\|E^{(1)}(h,t)\right\|_{L^2(\Omega)\times H^2(\Omega)}&\leq \lim_{h\rightarrow 0} M_0\left\|\frac{1}{h}\left(T(h)-I\right)\begin{pmatrix}\tilde{v}_0\\ \tilde{w}_0\end{pmatrix}-\mathcal{A}\begin{pmatrix}\tilde{v}_0\\ \tilde{w}_0\end{pmatrix}\right\|_{L^2(\Omega)\times H^2(\Omega)}\notag\\
		&:=\lim_{h\rightarrow 0}\Lambda_1(h)=0.\notag
	\end{align}
	\[\lim_{h\rightarrow 0}\frac{1}{h}\int_0^h\left\{T(h-s)\begin{pmatrix}G_0 \\ 0\end{pmatrix}\right\}ds=\begin{pmatrix}G_0 \\ 0\end{pmatrix}.\]
	Because $ G(\tilde{w})\in C([0, T]; H^2(\Omega))$, $\tilde{u}\in C\left([0, T]; H_0^1(\Omega)\right)$, then
	\[\lim_{h\rightarrow 0}\left\{\sup_{0\leq s\leq h}\left\|[G(\tilde{w})](s)-[G(\tilde{w})](0)\right\|_{H^2(\Omega)}+\beta_p\sup_{0\leq s\leq h}\left\|\tilde{u}(s)-\tilde{u}(0)\right\|_{H^1(\Omega)}\right\}=0,\]
	hence
	\begin{align}
		&\lim_{h\rightarrow 0}\left\|E^{(2)}(h,t)\right\|_{L^2(\Omega)\times H^2(\Omega)}\notag\\
		=&\lim_{h\rightarrow 0}\left\|\frac{1}{h}\int_0^h\left\{T(t+h-s)\begin{pmatrix}[G(\tilde{w})](s)+\beta_p\tilde{u}(s)\\ 0\end{pmatrix}\right\}ds-T(t)\begin{pmatrix}G_0\\ 0\end{pmatrix}\right\|_{L^2(\Omega)\times H^2(\Omega)}\notag\\
		=&\lim_{h\rightarrow 0}\left\|T(t)\frac{1}{h}\int_0^h\left\{T(h-s)\begin{pmatrix}[G(\tilde{w})](s)-G(\tilde{w}_0)+\beta_p(\tilde{u}(s)-\tilde{u}_0) \\ 0\end{pmatrix}\right\}ds\right\|_{L^2(\Omega)\times H^2(\Omega)}\notag\\
		\leq&\lim_{h\rightarrow 0}M_0\left\|\frac{1}{h}\int_0^h\left\{T(h-s)\begin{pmatrix}[G(\tilde{w})](s)-G(\tilde{w}_0)+\beta_p(\tilde{u}(s)-\tilde{u}_0)\\ 0\end{pmatrix}\right\}ds\right\|_{L^2(\Omega)\times H^2(\Omega)}\notag\\
		\leq&\lim_{h\rightarrow 0}M_0^2\sup_{0\leq s\leq h}\left\|[G(\tilde{w})](s)-G(\tilde{w}_0)+\beta_p(\tilde{u}(s)-\tilde{u}_0)\right\|_{L^2(\Omega)}\notag\\
		\leq&\lim_{h\rightarrow 0}M_0^2\left\{\sup_{0\leq s\leq h}\left\|[G(\tilde{w})](s)-[G(\tilde{w})](0)\right\|_{H^2(\Omega)}+\beta_p\sup_{0\leq s\leq h}\left\|\tilde{u}(s)-\tilde{u}(0)\right\|_{H^1(\Omega)}\right\}\notag\\
		:=&\lim_{h\rightarrow 0}\Lambda_2(h)=0.\notag
	\end{align}
	Define
	\[G_D(h,t):=\left[G(\tilde{w})\right](t+h)-\left[G(\tilde{w})\right](t)-\left[G'(\tilde{w})\right](t)\cdot\left[\tilde{w}(t+h)-\tilde{w}(t)\right],\]
	\[E^{(3)}_1(h,t):=\int_0^tT(t-s)\begin{pmatrix}\frac{1}{h}G_D(s)\\ 0\end{pmatrix}ds,\]
	\[E^{(3)}_2(h,t):=\int_0^tT(t-s)\begin{pmatrix}[G'(\tilde{w})](s)\left\{\frac{1}{h}\{\tilde{w}(s+h)-\tilde{w}(s)\}-\tilde{q}(s)\right\}\\0\end{pmatrix}ds,\]
	\[E^{(3)}_3(h,t):=\int_0^tT(t-s)\begin{pmatrix}\beta_p\left\{\frac{1}{h}[\tilde{u}(s+h)-\tilde{u}(s)]-[\mathcal{H}_2(\tilde{p}, \tilde{q})](s)\right\}\\ 0\end{pmatrix}ds.\]
	We then write
	\[E^{(3)}(h,t)=E^{(3)}_1(h,t)+E^{(3)}_2(h,t)+E^{(3)}_3(h,t).\]
	Denote $\tilde{v}_h(t)=\tilde{v}(t+h)-\tilde{v}(t)$, $\tilde{w}_h(t)=\tilde{w}(t+h)-\tilde{w}(t)$. Inequalities \eqref{Chap5Fre-Lip-continuity-v}, \eqref{Chap5Fre-Lip-continuity-w} from Lemma \ref{Chap5Fre-Lip-continuity} and \eqref{Chap5Bound-of-Fre-D} imply
	\begin{align}
		&\frac{1}{h}\big\{\left\|\left[D_v\mathbf{S}_o(\tilde{v}(t), \tilde{w}(t)+\theta_2)\right]\left(\tilde{v}_h(t)-\tilde{p}(t)\right)\right\|_{H^1(\Omega)}\notag\\
		&+\left\|\left[D_w\mathbf{S}_o(\tilde{v}(t), \tilde{w}(t)+\theta_2)\right]\left(\tilde{w}_h(t)-\tilde{q}(t)\right)\right\|_{H^1(\Omega)}\big\}
		\leq L_D\left\|E(h,t)\right\|_{L^2(\Omega)\times H^2(\Omega)}\label{Chap5U4}.
	\end{align}
	Here $L_D=C_o+C_o^*\left(\left\|\tilde{v}_0\right\|_{L^2(\Omega)}+\frac{\kappa}{2C}\right)$. Estimates  \eqref{Chap5Fre-uniformly-continuity-v} and \eqref{Chap5Fre-uniformly-continuity-w} from Lemma \ref{Chap5Fre-uniformly-continuity} and estimate \eqref{Chap5Lip-mild-solu-inquality} from Corollary \ref{Chap5Lip-mild-solu} give
	\begin{align}
		&\frac{1}{h}\big\{\left\|\left[D_v\mathbf{S}_o\left(\tilde{v}(t)+\tau\tilde{v}_h(t), \tilde{w}(t)+\tau \tilde{w}_h(t)+\theta_2\right)-D_v\mathbf{S}_o(\tilde{v}(t), \tilde{w}(t)+\theta_2)\right]\tilde{v}_h(t)\right\|_{H^1(\Omega)}\notag\\
		&+\left\|\left[D_w\mathbf{S}_o\left(\tilde{v}(t)+\tau \tilde{v}_h(t), \tilde{w}(t)+\tau \tilde{w}_h(t)+\theta_2\right)-D_w\mathbf{S}_o(\tilde{v}(t), \tilde{w}(t)+\theta_2)\right]\tilde{w}_h(t)\right\|_{H^1(\Omega)}\big\}\notag\\
		\leq&\left[\alpha_1(h)+\alpha_2(h)\right]L_V\label{Chap5U5}.
	\end{align}
	Thus \eqref{Chap5U4} and \eqref{Chap5U5}  imply
	\begin{align}
		&\left\|\frac{\tilde{u}(t+h)-\tilde{u}(t)}{h}-[\mathcal{H}_2(\tilde{p}, \tilde{q})](t)\right\|_{H^1(\Omega)}\notag\\
		=&\bigg\|\frac{1}{h}\int_0^1\left[D_v\mathbf{S}_o\left(\tilde{v}(t)+\tau \tilde{v}_h(t), \tilde{w}(t)+\tau w_h(t)+\theta_2\right)\right]\tilde{v}_h(t)\notag\\
		&+\left[D_w\mathbf{S}_o\left(\tilde{v}(t)+\tau \tilde{v}_h(t), \tilde{w}(t)+\tau \tilde{w}_h(t)+\theta_2\right)\right]\tilde{w}_h(t)d\tau-[\mathcal{H}_2(\tilde{p}, \tilde{q})](t)\bigg\|_{L^2(\Omega)}\notag\\
		\leq&\frac{1}{h}\big\{\left\|\left[D_v\mathbf{S}_o\left(\tilde{v}(t)+\tau \tilde{v}_h(t), \tilde{w}(t)+\tau \tilde{w}_h(t)+\theta_2\right)-D_v\mathbf{S}_o(\tilde{v}(t), \tilde{w}(t)+\theta_2)\right]\tilde{v}_h(t)\right\|_{H^1(\Omega)}\notag\\
		&+\left\|\left[D_w\mathbf{S}_o\left(\tilde{v}(t)+\tau \tilde{v}_h(t), \tilde{w}(t)+\tau \tilde{w}_h(t)+\theta_2\right)-D_w\mathbf{S}_o(\tilde{v}(t), \tilde{w}(t)+\theta_2)\right]\tilde{w}_h(t)\right\|_{H^1(\Omega)}\notag\\
		&+\left\|\left[D_v\mathbf{S}_o(\tilde{v}(t), \tilde{w}(t)+\theta_2)\right]\left(\tilde{v}_h(t)-\tilde{p}(t)\right)\right\|_{H^1(\Omega)}\notag\\
		&+\left\|\left[D_w\mathbf{S}_o(\tilde{v}(t), \tilde{w}(t)+\theta_2)\right]\left(\tilde{w}_h(t)-\tilde{q}(t)\right)\right\|_{H^1(\Omega)}\big\}\notag\\
		\leq&\left[\alpha_1(h)+\alpha_2(h)\right]L_V+L_D\left\|E(h,t)\right\|_{L^2(\Omega)\times H^2(\Omega)} \label{Chap5U6}.
	\end{align}
	Consequently, \eqref{Chap5U6} and the bounded property of the strongly continuous semigroup $T(t)$ imply
	\begin{align}
		\left\|E^{(3)}_3(h,t)\right\|_{L^2(\Omega)\times H^2(\Omega)}
		\leq& M_0\beta_p\int_0^t\left\|\frac{\tilde{u}(s+h)-\tilde{u}(s)}{h}-[\mathcal{H}_2(\tilde{p}, \tilde{q})](s)\right\|_{H^1(\Omega)}ds\notag\\
		\leq&M_0\beta_p\left(T_0\left[\alpha_1(h)+\alpha_2(h)\right]L_V+L_D\int_0^t\left\|E(h,s)\right\|_{L^2(\Omega)\times H^2(\Omega)}ds\right)\notag\\
		:=&\Lambda_3(h)+M_0\beta_pL_D\int_0^t\left\|E(h,s)\right\|_{L^2(\Omega)\times H^2(\Omega)}ds\notag\\
		\rightarrow&0+M_0\beta_pL_D\int_0^t\left\|E(h,s)\right\|_{L^2(\Omega)\times H^2(\Omega)}ds,\ \text{as}\ h\rightarrow 0\notag.
	\end{align}
	Using  the estimate \eqref{Lip-G-2} of Fr\'{e}chet derivative $G'\left(\tilde{w}\right)$ from Corollary \ref{Lip-G-Lem} again gives
	\begin{align}
		\left\|E^{(3)}_2(h,t)\right\|_{L^2(\Omega)\times H^2(\Omega)}\leq M_0L_G\int_0^t\left\|E(h,s)\right\|_{L^2(\Omega)\times H^2(\Omega)}ds\notag.
	\end{align}
	Since $G_D(h,t)\in H^2(\Omega)$, the estimate \eqref{Chap5Lip-mild-solu-inquality} of Corollary \ref{Chap5Lip-mild-solu} implies the function $\tilde{w}$ is Lipschitz continuous with respect to $t\in[0, T]$.  Employing this fact and the limit \eqref{uniformly-continuous-Fre-G} of Corollary \ref{Lip-G-Lem} gives
	
	\begin{align}
		&\left\|E^{(3)}_1(h,t)\right\|_{L^2(\Omega)\times H^2(\Omega)}\notag\\
		\leq&T_0M_0\sup_{\begin{smallmatrix}0\leq t<t+h\leq T\end{smallmatrix}}\frac{1}{h}\left\|G_D(h,t)\right\|_{L^2(\Omega)}\notag\\
		\leq&T_0M_0\sup_{\begin{smallmatrix}0\leq t<t+h\leq T\end{smallmatrix}}\frac{1}{h}\left\|G_D(h,t)\right\|_{H^2(\Omega)}\notag\\
		=&{\textstyle \frac{T_0M_0}{h}}\sup_{\begin{smallmatrix}t\in[0, T]\\ t+h\in[0, T]\end{smallmatrix}}{\textstyle \left\|\int_0^1\left[G'(\tilde{w}(t)+\tau[\tilde{w}(t+h)-\tilde{w}(t)])-G'(\tilde{w}(t))\right]\left[\tilde{w}(t+h)-\tilde{w}(t)\right]d\tau\right\|_{ H^2(\Omega)}}\notag\\
		\leq& { \textstyle\frac{T_0M_0L_V h}{h}}\sup_{0\leq t< t+ h\leq T}{\textstyle \left\|\int_0^1G'(\tilde{w}(t)+\tau[\tilde{w}(t+h)-\tilde{w}(t)])-G'(\tilde{w}(t))d\tau\right\|_{\mathcal{B}\left(H_o^2(\Omega)\right)}}\notag\\
		=&T_0M_0L_V\sup_{\begin{smallmatrix}0\leq t\leq t+\tau h\leq T\\ 0\leq\tau\leq1\end{smallmatrix}}\left\|G'(\tilde{w}(t)+\tau[\tilde{w}(t+h)-\tilde{w}(t)])-G'(\tilde{w}(t))\right\|_{\mathcal{B}\left(H_o^2(\Omega)\right)}\notag\\
		:=&\Lambda_4(h)\rightarrow 0 ,\ \text{as}\ h\rightarrow 0. \notag
	\end{align}
	Summing up, by setting $\pi_v=M_0\beta_pL_D+M_0L_G$,  we have shown
	\[\left\|E(h,t)\right\|_{L^2(\Omega)\times H^2(\Omega)}\leq\Lambda_1(h)+\Lambda_2(h)+\Lambda_3(h)+\Lambda_4(h)+\pi_v\int_0^t\left\|E(h,s)\right\|_{L^2(\Omega)\times H^2(\Omega)}ds.\]
	Gronwall's inequality thus implies the inequality
	\[\left\|E(h,t)\right\|_{L^2(\Omega)\times H^2(\Omega)}\leq\left(\Lambda_1(h)+\Lambda_2(h)+\Lambda_3(h)+\Lambda_4(h)\right)e^{t\pi_v}\]
	holds for $t\in [0, T]$. Letting $h\rightarrow 0^+$, we then deduce that the $\left(\tilde{v}, \tilde{w}\right)$ is differentiable from the right and the right  derivative of $\left(\tilde{v}, \tilde{w}\right)$ coincides with $\left(\tilde{p}, \tilde{q}\right)$. Because $\left(\tilde{p}, \tilde{q}\right)$ is continuous on $[0, T]$, by using Lemma \ref{time-derivative-continuity}, we conclude $\left(\tilde{v}, \tilde{w}\right)\in C^1\left([0, T]; L^2(\Omega)\times H_o^2(\Omega)\right)$.
	
	Lemma \ref{Chap5Fre-Lip-continuity}, estimates \eqref{Chap5Bound-of-Fre-D} and Lemma \ref{Chap5Fre-uniformly-continuity} imply $\tilde{u}=\mathbf{S}_o(\tilde{v}, \tilde{w}+\theta_2)\in C^1\left([0, T]; L^2(\Omega)\right)$, whereby $\left(G\left(\tilde{w}\right)+\beta_p\tilde{u}, 0\right)\in C^1\left([0, T]; L^2(\Omega)\times H_o^2(\Omega)\right)$. 
	
	By Lemma \ref{IEE-S}, the mild solution $\left(\tilde{v}, \tilde{w}\right)$, defined by \eqref{Chap5mild-solu-form}, uniquely solves the semilinear evolution equation \eqref{Chap54th-SWE-1} on $[0, T]$, $\left(\tilde{v}, \tilde{w}\right)$ is a unique strict solution of semilinear evolution equation \eqref{Chap54th-SWE-1}, and\\
	$\left(\tilde{v}, \tilde{w}\right)\in C\left([0, T]; H_o^2(\Omega)\times H_o^4(\Omega)\right)\cap C^1\left([0, T]; L^2(\Omega)\times H_o^2(\Omega)\right)$. Thus, $\tilde{w}$ uniquely solves the equation \eqref{Chap54th-LWE-1} and $\tilde{v}(x, t)=\frac{\partial\tilde{w}}{\partial t}(x,t)$, $(x,t)\in\Omega\times [0, T]$, for all $T\in(0, T_0)$.
\end{proof}
\appendix

\section{Proofs  in Section \ref{Sec:prerequisite}} \label{AppA}
\subsection{Proof of Corollary \ref{estimates}}
\begin{proof}
	We first prove assertion \eqref{w-lower-bound} of Corollary \ref{estimates}.
	
	Since $ w\in B_r\left({w}_0, T\right)$,  then $\|w(t)-w_0\|_{H^{2}(\Omega)}\leq r$ holds for all $t\in [0, T]$.
	
	According to the triangle inequality and the Sobolev embedding theorem, there exists a constant $C=C(\Omega)$, such that for all $r\in\left(0, \frac{\kappa}{2C}\right)$, it follows that
	\bse\label{w-lower-bound1}
	\be\label{w-lower-bound11}
	w(t)=w_0+w(t)-w_0\geq\kappa-\left\|w(t)-w_0\right\|_{L^{\infty}(\Omega)}\geq\kappa-C\left\|w(t)-w_0\right\|_{H^{2}(\Omega)}\geq\kappa-Cr\geq\frac{\kappa}{2},
	\ee
	\be\label{w-lower-bound12}
	\left\|w(t)\right\|_{H^2(\Omega)}\leq\tilde{C},
	\ee
	\ese
	hold for all $t\in [0, T]$. Here, $\tilde{C}=\frac{\kappa}{2C}+\left\|w_0\right\|_{H^2(\Omega)}$.  This proves assertion \eqref{w-lower-bound}.
	
	According to \eqref{w-lower-bound1} and $r\in\left(0, \frac{\kappa}{2C}\right)$, we have
	\begin{align}
		\sup_{t\in[0, T]}\left[\int_\Omega\left|\frac{1}{w(t)}\right|^2dx\right]\leq \frac{4C}{\kappa^2},\ \sup_{t\in[0, T]}\left[\int_\Omega\left|\nabla\left[\frac{1}{w(t)}\right]\right|^2dx\right]&=\sup_{t\in[0, T]}\left[\int_\Omega\frac{\left|\nabla w(t) \right|^2}{\left|w(t)\right|^4}dx\right]\notag\\
		&\leq\frac{16}{\kappa^4}\sup_{t\in[0, T]}\left[\int_\Omega\left|\nabla w(t)\right|^2dx\right]\notag\\
		&\leq\frac{16}{\kappa^4}\sup_{t\in[0, T]}\left\|w(t)\right\|_{H^2(\Omega)}^2\notag\\
		&\leq\frac{16}{\kappa^4}\tilde{C}^2\notag.
	\end{align}
	\begin{align}
		\sup_{t\in[0, T]}\left\{\int_\Omega \left|\Delta\left[\frac{1}{w(t)}\right]\right|^2dx\right\}
		\leq&\sup_{t\in[0, T]}\left[\left\|\frac{\Delta w(t)}{[w(t)]^2}\right\|_{L^2(\Omega)}
		+\left\|\frac{2\left(\nabla w(t)\right)^2}{[w(t)]^3}\right\|_{L^2(\Omega)}\right]^2\notag\\
		\leq&\sup_{t\in[0, T]}\left[\frac{4}{\kappa^2}\left\|\Delta w(t)\right\|_{L^2(\Omega)}+\frac{16}{\kappa^3}\left\|\left(\nabla w(t)\right)^2\right\|_{L^2(\Omega)}\right]^2\notag\\
		\leq&\sup_{t\in[0, T]}\left[\frac{4}{\kappa^2}\left\|w(t)\right\|_{H^2(\Omega)}+\frac{16}{\kappa^3}\left\|\nabla w(t)\right\|^2_{L^{4}(\Omega)}\right]^2\notag\\
		\leq&\sup_{t\in[0, T]}\left[\frac{4}{\kappa^2}\left\|w(t)\right\|_{H^2(\Omega)}+\frac{16C}{\kappa^3}\left\|\nabla w(t)\right\|^2_{H^{1}(\Omega)}\right]^2\notag\\
		\leq&\sup_{t\in[0, T]}\left[\frac{4}{\kappa^2}+\frac{16C}{\kappa^3}\left\|w(t)\right\|_{H^2(\Omega)}\right]^2\left\|w(t)\right\|^2_{H^2(\Omega)}\notag\\
		\leq&\left[\frac{4}{\kappa^2}+\frac{16C}{\kappa^3}\tilde{C}\right]^2\tilde{C}^2\notag.
	\end{align}
	\begin{align}
		\sup_{t\in[0, T]}\left\|\frac{1}{w(t)}\right\|^2_{H^2(\Omega)}=&\sup_{t\in[0, T]}\left\{\int_\Omega\left|\frac{1}{w(t)}\right|^2+\left|\nabla\left[\frac{1}{w(t)}\right]\right|^2+\left|\Delta\left[\frac{1}{w(t)}\right]\right|^2dx\right\}\notag\\
		\leq&\frac{4C}{\kappa^2}+\frac{16}{\kappa^4}\tilde{C}^2+\left[\frac{4}{\kappa^2}+\frac{16C}{\kappa^3}\tilde{C}\right]^2\tilde{C}^2\label{w-1-1}
	\end{align}
	We set $C_1^2=\frac{4C}{\kappa^2}+\frac{16}{\kappa^4}\tilde{C}^2+\left[\frac{4}{\kappa^2}+\frac{16C}{\kappa^3}\tilde{C}\right]^2\tilde{C}^2$,
	$C_1$ is a positive constant depending on $\Omega$, $\kappa$,  and $\left\|w_0\right\|_{H^2(\Omega)}$. Because $H^2(\Omega)$ is an algebra, i.e. Lemma \ref{alg}, the assertion \eqref{C-a} of Corollary \ref{estimates} holds for $t\in [0, T]$.
	With these facts, we continue on to show the assertions \eqref{C-d} of Corollary \ref{estimates}. For all $w_1$ and $w_2\in B_r(w_0,  T)$, the algebraic property of $H^2(\Omega)$ from Lemma \ref{alg} and the triangle inequality imply
	\[\sup_{t\in[0, T]}\left\| \frac{\left[w_1(t)+w_2(t)\right]}{[w_1(t)]^{2}[w_2(t)]^{2}}\right\|_{H^2(\Omega)}\leq2 C^3_1,\
	\sup_{t\in[0, T]}\left\|\frac{[w_1(t)]^2+[w_2(t)]^2+w_1(t) w_2(t)}{[w_1(t)]^3 [w_2(t)]^3}\right\|_{H^2(\Omega)}\leq3C_1^4.\]
	Setting $C_2=2 C^3_1$ and $C_3=3C^4_1$, hence we deduce \eqref{C-d} of Corollary \ref{estimates}.
	
	This concludes the proof of Corollary \ref{estimates}.
\end{proof}

\subsection{Proof of Corollary \ref{Lip-G-Lem}}
\begin{proof}
	Recall that $r\in\left(0, \frac{\kappa}{2C}\right)$, $w_0=\tilde{w}_0+\theta_2$, $\kappa=\min_{\overline{\Omega}}w_0>0$, $w(t)=\tilde{w}(t)+\theta_2$, $\tilde{w}\in B_r(\tilde{w}_0, T)$. For small $h\in(0, T)$ such that $t+h\in(0, T]$, $\left\|\tilde{w}(t+h)-\tilde{w}_0\right\|_{H^2(\Omega)}\leq r$, \eqref{C-a} and \eqref{C-d} of Corollary \ref{estimates} imply  \eqref{Holdercontinuous} and \eqref{Lip-G} of Corollary \ref{Lip-G-Lem} are valid with $L_G=\beta_FC_2$.
	
	In particular, for $\tilde{w}_1\in B_r\left(\tilde{w}_0,  T\right)$, set $\tilde{w}_2(t)=\tilde{w}_0$, $\forall\ t\in[0, T]$, then $[G(\tilde{w}_2)](t)=G(\tilde{w}_0)$. Hence \eqref{Lip-G-1} of Corollary \ref{Lip-G-Lem} is valid from the assertion \eqref{Lip-G} of Corollary \ref{Lip-G-Lem}.
	
	Set $\tilde{w}_2=\tilde{w}\in B_r\left(\tilde{w}_0, T\right)$, for $q\in C\left([0, T]; H_o^2(\Omega)\right)$, choose small $\lambda\in\mathbb{R}$, such that
	\[\tilde{w}_1=\tilde{w}+\lambda q\in B_r\left(\tilde{w}_0, T\right).\]
	Then the Fr\'{e}chet derivative $G'\left(\tilde{w}\right)q$ of $G$ with respect to $\tilde{w}\in B_r\left(\tilde{w}_0, T\right) $ exists as a linear operator  $G'(\tilde{w}):\ C\left([0, T]; H_o^2(\Omega)\right)\longrightarrow C\left([0, T]; H_o^2(\Omega)\right)$ given by
	\[G'\left(\tilde{w}\right)q=\lim_{\lambda\rightarrow0}\frac{1}{\lambda}\left[G\left(\tilde{w}+\lambda q\right)-G\left(\tilde{w}\right)\right]=\frac{2\beta_F}{\left(\tilde{w}+\theta_2\right)^3}q,\] \[\left[G'\left(\tilde{w}\right)q\right](t)=\left[G'\left(\tilde{w}(t)\right)\right]q(t)=\frac{2\beta_F}{\left(\tilde{w}(t)+\theta_2\right)^3}q(t).\]
	According to the assertion \eqref{Lip-G}, the inequality \eqref{Lip-G-2} holds by the following computation:
	\begin{align}
		\sup_{t\in[0, T]}\left\|\left[G'\left(\tilde{w}\right)q\right](t)\right\|_{H^2(\Omega)}=&\sup_{t\in[0, T]}\left\|\lim_{\lambda\rightarrow0}\frac{[G\left(\tilde{w}+\lambda q\right)](t)-[G\left(\tilde{w}\right)](t)}{\lambda}\right\|_{H^2(\Omega)}\notag\\
		=&\lim_{\lambda\rightarrow0}\frac{1}{\lambda}\sup_{t\in[0, T]}\left\|[G\left(\tilde{w}_1\right)](t)-[G\left(\tilde{w}_2\right)](t)\right\|_{H^2(\Omega)}\notag\\
		\leq&\lim_{\lambda\rightarrow0}\frac{1}{\lambda}L_G\sup_{t\in[0, T]}\left\| \tilde{w}_1(t)-\tilde{w}_2(t)\right\|_{H^{2}(\Omega)}\notag\\
		=&L_G\sup_{t\in[0, T]}\left\|q(t)\right\|_{H^{2}(\Omega)}\notag.
	\end{align}
	For all $t\in[0, T]$, choose small $h\in\left(0, T\right)$ and $\tau\in[0, 1]$ such that $t+h\in(0, T]$,
	\[\left\|\tilde{w}(t)+\tau[\tilde{w}(t+h)-\tilde{w}(t)]-\tilde{w}_0\right\|_{H^2(\Omega)}\leq r,\]
	then for $\psi\in H_o^2(\Omega)$ with $\|\psi\|_{H^2(\Omega)}\leq1$, $G'(\tilde{w}(t)): \psi\in H_o^2(\Omega)\longrightarrow \left[G'(\tilde{w}(t))\right]\psi\in H_o^2(\Omega)$. By the algebraic properties of $H^2(\Omega)$, i.e. \eqref{C-a} and \eqref{C-d} from Corollary \ref{estimates}, we have
	\begin{align}
		&\left\|G'(\tilde{w}(t)+\tau[\tilde{w}(t+h)-\tilde{w}(t)])-G'(\tilde{w}(t))\right\|_{\mathcal{B}\left(H_o^2(\Omega)\right)}\notag\\
		=&\left\|\left[G'(\tilde{w}(t)+\tau[\tilde{w}(t+h)-\tilde{w}(t)])\right]\psi-\left[G'(\tilde{w}(t))\right]\psi\right\|_{H^2(\Omega)}\notag\\
		\leq&2\beta_F\sup_{\begin{smallmatrix}0\leq t< t+h\leq T\\ 0\leq \tau\leq 1\end{smallmatrix}}\left\|\frac{1}{\left({w}(t)+\tau[{w}(t+h)-{w}(t)]\right)^3}-\frac{1}{[{w}(t)]^3}\right\|_{H^2(\Omega)}\|\psi\|_{H^2(\Omega)}\notag\\
		\leq&2\beta_FC_3\sup_{0\leq t<t+h\leq T}\|\tilde{w}(t+h)-\tilde{w}(t)\|_{H^2(\Omega)}\notag.
	\end{align}
	Since  $\tilde{w}\in C\left([0, T]; H_o^2(\Omega)\right)$,  $\tilde{w}$ are uniformly continuous with respect to $t\in[0, T]$, hence the assertion \eqref{uniformly-continuous-Fre-G} is proved by
	\[\lim_{h\rightarrow 0^+}\sup_{0\leq t< t+ h\leq T}\|\tilde{w}(t+\tau h)-\tilde{w}(t)\|_{H^2(\Omega)}=0.\]
	This concludes the proof of Corollary \ref{Lip-G-Lem}.
	
\end{proof}
\section{Proofs in Section \ref{Chap54th-order problem}}\label{AppB}
\subsection{Proof of Lemma \ref{generator}}
\begin{proof}
	We aim to show that $\mathcal{A}$, defined by \eqref{A-1}, is skew adjoint on the Hilbert space $\mathfrak{X}$ defined by \eqref{state-space}, and thus generates a strongly continuous semigroup ($C_0$-semigroup) on $\mathfrak{X}$ by using Stone's Lemma (see 3.24 Theorem, Section 3, Chapter II, \cite{EK}).
	
	From the definition \eqref{A-1} of $\mathcal{A}$, $\mathcal{A}$ is densely defined in $\mathfrak{X}$, i.e. $\overline{D(\mathcal{A})}=\mathfrak{X}$,  then $\mathcal{A}$ is skew symmetric (i.e. $\mathbf{i}\mathcal{A}$ is symmetric) for any two $(\phi_1, \phi_2)\in D(\mathcal{A})$ and $(\psi_1, \psi_2 )\in D(\mathcal{A})$ by the following computations:
	\begin{align}
		\left\langle\mathcal{A}\begin{pmatrix}\phi_1\\ \phi_2\end{pmatrix}, \begin{pmatrix}\psi_1\\ \psi_2 \end{pmatrix}\right\rangle_{\mathfrak{X}}&=\left\langle\begin{pmatrix}\mathbb{A}\phi_2\\ \phi_1 \end{pmatrix}, \begin{pmatrix}\psi_1\\ \psi_2 \end{pmatrix}\right\rangle_{\mathfrak{X}}\notag\\
		&=\int_\Omega\mathbb{A}\phi_2\cdot\overline{\psi_1}+\nabla\phi_1\cdot\nabla\overline{\psi_2}+\Delta\phi_1\cdot\Delta\overline{\psi_2}dx\notag\\
		&=\int_\Omega -\nabla\phi_2\cdot\nabla\overline{\psi_1}-\Delta\phi_2\cdot\Delta\overline{\psi_1}+\nabla\phi_1\cdot\nabla\overline{\psi_2}+\Delta\phi_1\cdot\Delta\overline{\psi_2}dx\notag\\
		&=-\left[\int_\Omega \nabla\phi_2\cdot\nabla\overline{\psi_1}+\Delta\phi_2\cdot\Delta\overline{\psi_1}-\nabla\phi_1\cdot\nabla\overline{\psi_2}-\Delta\phi_1\cdot\Delta\overline{\psi_2}dx\right]\notag\\
		&=-\left[\int_\Omega \nabla\phi_2\cdot\nabla\overline{\psi_1}+\Delta\phi_2\cdot\Delta\overline{\psi_1}+\phi_1\cdot\overline{\mathbb{A}\psi_2}dx\right]\notag\\
		&=-\left\langle\begin{pmatrix}\phi_1\\ \phi_2 \end{pmatrix}, \begin{pmatrix}\mathbb{A}\psi_2\\ \psi_1 \end{pmatrix}\right\rangle_{\mathfrak{X}}\notag\\
		&=-\left\langle\begin{pmatrix}\phi_1\\ \phi_2 \end{pmatrix}, \mathcal{A}\begin{pmatrix}\psi_1\\ \psi_2 \end{pmatrix}\right\rangle_{\mathfrak{X}}\notag
	\end{align}
	Furthermore, $\text{Re}\left\langle\mathcal{A}\begin{pmatrix}\psi\\ \phi\end{pmatrix}, \begin{pmatrix}\psi\\ \phi \end{pmatrix}\right\rangle_{\mathfrak{X}}=0$ for all $(\psi, \phi )\in D(\mathcal{A})$, so $\mathcal{A}$ is dissipative. By using the Lax-Milgram Theorem (Theorem 1, Section 6.2, \cite{EL}), we have the inverse $\mathbb{A}^{-1}$ of $\mathbb{A}$ exists, thus we define an operator
	\[\mathcal{R}=\begin{pmatrix}0 &1\\ \mathbb{A}^{-1} &0\end{pmatrix}.\]
	Then
	\[\mathcal{R}\mathfrak{X}\subset D(\mathcal{A}),\ \mathcal{A}\mathcal{R}=I, \ \mathcal{R}\mathcal{A}\begin{pmatrix}\psi\\ \phi \end{pmatrix}=\begin{pmatrix}\psi\\ \phi \end{pmatrix},\ \forall\ \left(\psi, \phi \right)\in D(\mathcal{A}).\]
	Therefore, $\mathbf{i}\mathcal{A}$ is invertible and the resolvent set $\rho(\mathbf{i}\mathcal{A})$ of $\mathbf{i}\mathcal{A}$ satisfies $\rho(\mathbf{i}\mathcal{A})\cap \mathbb{R}\neq\emptyset$, so the spectrum $\sigma(\mathbf{i}\mathcal{A})\subseteq\mathbb{R}$, consequently, $\mathbf{i}\mathcal{A}$ is selfadjoint, as a result, $\mathcal{A}$ is skew adjoint. According to Stone's Lemma, we have the linear operator $\mathcal{A}$ generates a $C_0$-semigroup
	\[\left\{T(t)\in \mathcal{B}(\mathfrak{X}):\ t\in[0, \infty)\right\}.\]
\end{proof}

\bibliographystyle{abbrv}
\bibliography{Bibliography4}

\end{document}